\documentclass{amsart}

\usepackage{paralist}
\usepackage[]{hyperref}
\usepackage{color}
\usepackage{pgf,tikz, tikzscale}
\usetikzlibrary{decorations.pathreplacing,decorations.markings}
\usetikzlibrary{arrows, knots, external, positioning}

\usetikzlibrary{arrows.meta}
\tikzstyle arrowstyle=[scale=2]
\tikzstyle mdirected=[postaction={decorate,decoration={markings,
    mark=at position.5 with {\arrow[arrowstyle]{stealth}}}}]
\tikzstyle edirected=[postaction={decorate,decoration={markings,
    mark=at position 1.0 with {\arrow[arrowstyle]{stealth}}}}]

\tikzset{bullet/.style={
shape = circle,fill = black, inner sep = 0pt, outer sep = 0pt, minimum size = 0.35em, line width = 0pt, draw=black!100}}

\usetikzlibrary{circuits.logic.US,circuits.logic.IEC,fit}

\usepackage{tabularx}
\usepackage{array}
\newcolumntype{C}[1]{>{\centering\arraybackslash}m{#1}}

\newtheorem{thm}{Theorem}[section]
\newtheorem{lem}[thm]{Lemma}
\newtheorem{claim}[thm]{Claim}
\newtheorem{cor}[thm]{Corollary}
\newtheorem{prop}[thm]{Proposition}

\theoremstyle{definition}
\newtheorem{defn}{Definition}[section]

\numberwithin{equation}{section}

\newtheorem{theorem}{Theorem}[section]

\newtheorem{lemma}[theorem]{Lemma}

\theoremstyle{definition}

\theoremstyle{definition}
\newtheorem{remark}{Remark}[section]
\theoremstyle{definition}

\numberwithin{equation}{section}

\newcommand{\RN}[1]{%
  \textup{\uppercase\expandafter{\romannumeral#1}}%
}
\makeatletter
\let\@wraptoccontribs\wraptoccontribs
\makeatother

\begin{document}


\title[Symplectic fillings of Seifert 3-manifolds]{A topological characterization of symplectic fillings of Seifert 3-manifolds}


\author{Hakho Choi}
\address{Hakho Choi: Center for Quantum Structures in Modules and Spaces, Seoul National University, Seoul 08826, Republic of Korea }
\email{hako85@snu.ac.kr}

\author{Jongil Park}
\address{Jongil Park: Department of Mathematical Sciences, Seoul National University, Seoul 08826,
 Republic of Korea }
\email{jipark@snu.ac.kr}

\address{JaeKwan Jeon: Department of Mathematics, Chungnam National University, Daejeon 34134, Republic of Korea }
\email{jk-jeon@cnu.ac.kr}


\thanks{}
\subjclass[2010]{}%
\keywords{Rational blowdown surgery, Seifert $3$-manifold, symplectic filling}
\date{July 26, 2022; revised at January 23, 2023}

\begin{abstract}
In this paper, we investigate a surgical interpretation for minimal symplectic fillings of a given Seifert 3-manifold equipped with a canonical contact structure. Consequently, we determine a necessary and sufficient condition for a minimal symplectic filling of a Seifert 3-manifold satisfying certain conditions to be obtained by a sequence of rational blowdown surgery from the minimal resolution of the corresponding weighted homogeneous surface singularity. Furthermore, as an application, we prove that every minimal symplectic filling of a large family of Seifert 3-manifolds with a canonical contact structure is in fact realized as a Milnor fiber of the corresponding weighted homogeneous surface singularity in the Appendix. 
\end{abstract}

\maketitle
\hypersetup{linkcolor=black}

\section{Introduction}
 A fundamental problem in symplectic 4-manifold topology is the classification of symplectic fillings of certain 3-manifolds equipped with a natural contact structure. Researchers have long studied the symplectic fillings of the link of a normal complex surface singularity. Note that Seifert 3-manifolds can be viewed as a link of weighted homogeneous surface singularities, and the link of such a normal surface singularity carries a canonical contact structure, also known as the Milnor fillable contact structure.
 For example, P.~Lisca~\cite{Lis}, M.~Bhupal and K.~Ono~\cite{BOn}, and the second author of this study et al.~\cite{PPSU} completely classified all minimal symplectic fillings of lens spaces and certain small Seifert 3-manifolds coming from the link of quotient surface singularities.

 Topologists working on 4-manifold topology are also interested in finding a surgical interpretation for the symplectic fillings of a given 3-manifold. More specifically, topologists investigate whether a surgical description of these fillings exists. 
 Indeed, a \emph{rational blowdown} surgery, introduced by R.~Fintushel and R.~Stern~\cite{FS} and generalized by the second author~\cite{Par} and A.~Stipsicz, Z.~Szab{\'o} and J.~Wahl~\cite{SSW}, is a powerful tool used in these investigations.  
For example, for the link of quotient surface singularities equipped with a canonical contact structure, it has been proven~\cite{BOz}, \cite{CP1} that every minimal symplectic filling is obtained by a sequence of rational blowdowns from the minimal resolution of the singularity. However, L.~Starkston~\cite{Sta2} showed that the symplectic fillings of some Seifert 3-manifolds cannot be obtained by a sequence of rational blowdowns from the minimal resolution of the corresponding singularity. 
Hence, knowing which Seifert 3-manifolds have a rational blowdown surgery interpretation for their minimal symplectic fillings is an intriguing question.

 In this paper, we first investigate a relation between rational blowdown surgery and the minimal symplectic fillings of a given Seifert 3-manifold with a canonical contact structure, so that we determine a necessary and sufficient condition for a minimal symplectic filling of a given Seifert 3-manifold satisfying certain conditions to be obtained by a sequence of rational blowdowns from the minimal resolution of the corresponding weighted homogeneous surface singularity.  
 In general, a Seifert 3-manifold can be considered as an $S^1$-fibration over a Riemann surface and it may have any number of singular fibers. In this article, we only consider a Seifert $3$-manifold $Y$ as an $S^1$-fibration over the $2$-sphere such that it can be described by $Y(-b; (\alpha_1, \beta_1), (\alpha_2, \beta_2),\ldots (\alpha_n, \beta_n))$, whose Dehn surgery diagram is given in Figure~\ref{Seifert} and given as a boundary of a plumbing 4-manifold of disk bundles over a $2$-sphere according to the graph $\Gamma$ in Figure~\ref{Seifert}.
The integers $b_{ij}\geq 2$ are uniquely determined by the following continued fraction:
 $$\frac{\alpha_i}{\beta_i}=[ b_{i1}, b_{i2}, \dots, b_{ir_i} ]=b_{i1}-\displaystyle {\frac{1}{b_{i2}-\displaystyle\frac{1}{\cdots-\displaystyle\frac{1}{b_{ir_i}}}}}$$

\begin{figure}[h]
\begin{tikzpicture}[scale=0.5]
\begin{scope}
\begin{knot}[
	clip width=5,
	clip radius = 2pt,
	end tolerance = 1pt,
]
\strand (0,0) ellipse (3 and 1.5);
\strand (-1.5,-1.75) ellipse (0.25 and 0.8);
\strand (-.2,-1.75) ellipse (0.25 and 0.8);
\draw (.8,-2) node[below] {$\cdots$};
\strand (1.75,-1.75) ellipse (0.25 and 0.8);
\draw (-3,1.5) node {$-b$};
\draw (-1.5,-2.55) node[below] {$-\frac{\alpha_1}{\beta_1}$};
\draw (-.2,-2.55) node[below] {$-\frac{\alpha_2}{\beta_2}$};
\draw (1.75,-2.55) node[below] {$-\frac{\alpha_n}{\beta_n}$};

\flipcrossings{1,4,5}
\end{knot}
\end{scope}
\begin{scope}[shift={(8,0)}]
\node[bullet] at (0,1.5){};
\node[bullet] at (0,0){};
\node[bullet] at (-2,0){};
\node[bullet] at (3,0){};
\node[bullet] at (0,-1){};
\node[bullet] at (-2,-1){};
\node[bullet] at (3,-1){};
\node[bullet] at (0,-3){};
\node[bullet] at (-2,-3){};
\node[bullet] at (3,-3){};
\draw (0,1.5)--(0,-1.5);
\draw (0,1.5)--(3,0)--(3,-1.5);
\draw (0,1.5)--(-2,0)--(-2,-1.5);
\draw[dotted,thick](1.25,0)--(1.75,0);
\draw[dotted,thick](1.25,-1)--(1.75,-1);
\draw[dotted,thick](1.25,-3)--(1.75,-3);

\draw[dotted](0,-1.5)--(0,-2.5);
\draw[dotted](3,-1.5)--(3,-2.5);
\draw[dotted](-2,-1.5)--(-2,-2.5);
\draw(0,-2.5)--(0,-3);
\draw(3,-2.5)--(3,-3);
\draw(-2,-2.5)--(-2,-3);
\draw (0,1.5) node[above] {$-b$};
\draw (0,0) node[left] {$-b_{21}$};
\draw (0,-1) node[left] {$-b_{22}$};
\draw (0,-3) node[left] {$-b_{2r_2}$};
\draw (-2,0) node[left] {$-b_{11}$};
\draw (-2,-1) node[left] {$-b_{12}$};
\draw (-2,-3) node[left] {$-b_{1r_1}$};
\draw (3,0) node[right] {$-b_{n1}$};
\draw (3,-1) node[right] {$-b_{n2}$};
\draw (3,-3) node[right] {$-b_{nr_n}$};
\end{scope}

\end{tikzpicture}
\caption{\mbox{Surgery diagram of $Y$ and its associated plumbing graph $\Gamma$}}
\label{Seifert}
\end{figure}
 




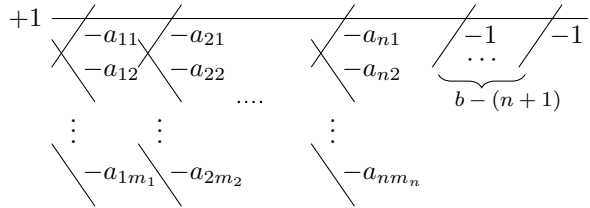
\begin{figure}[h]
\begin{tikzpicture}[scale=0.5]

\begin{scope}[xscale=2.3,yscale=1.85, shift={(6.6,1)}]

\draw(0,0)--(6.2,0);
\draw (0.5,0.2)--(0,-0.7);
\draw (1.5,0.2)--(1,-0.7);
\draw (3.5,0.2)--(3,-0.7);
\draw (4.9,0.2)--(4.4,-0.7);
\draw (5.9,0.2)--(5.4,-0.7);
\node at (5.,-0.6) {$\cdots$};
\draw[dotted,thick](2.15,-1.2)--(2.45,-1.2);
	\draw [decorate,decoration={brace,amplitude=5pt,mirror},xshift=2pt,yshift=-3pt]
	(4.4,-0.7) -- (5.4,-0.7) node [black,midway,xshift=10pt,yshift=-10pt] 
	{\footnotesize $b-(n+1)$};

\draw (0,-0.3)--(0.5,-1.2);
\draw (1,-0.3)--(1.5,-1.2);
\draw (3,-0.3)--(3.5,-1.2);

\node at (0.25, -1.5) {$\vdots$};
\node at (1.25, -1.5) {$\vdots$};
\node at (3.25, -1.5) {$\vdots$};

\draw (0,-1.8)--(0.5,-2.7);
\draw (1,-1.8)--(1.5,-2.7);
\draw (3,-1.8)--(3.5,-2.7);

\node[left] at (0,0) {$+1$};
\node[right] at (0.25, -0.25) {$-a_{11}$};
\node[right] at (1.25, -0.25) {$-a_{21}$};
\node[right] at (3.25, -0.25) {$-a_{n1}$};
\node[right] at (4.65, -0.25) {$-1$};
\node[right] at (5.65, -0.25) {$-1$};

\node[right] at (0.25, -0.75) {$-a_{12}$};
\node[right] at (1.25, -0.75) {$-a_{22}$};
\node[right] at (3.25, -0.75) {$-a_{n2}$};

\node[right] at (0.25, -2.25) {$-a_{1m_1}$};
\node[right] at (1.25, -2.25) {$-a_{2m_2}$};
\node[right] at (3.25, -2.25) {$-a_{nm_n}$};

\end{scope}

\end{tikzpicture}
\caption{Concave cap $K$ }
\label{X}
\end{figure}

We introduce the main results by starting with a minimal symplectic filling $W$ of a Seifert $3$-manifold $Y$ with a canonical contact structure.
While $b\geq(n+1)$, we obtain a closed rational symplectic $4$-manifold $M=W\cup K$ by gluing a concave cap $K$  to $W$ along $Y$ (refer to Figure~\ref{X}). Then, the image of $K$ under blowing-downs from $M$ to $\mathbb{CP}^2$ is called a \emph{symplectic line arrangement} $S\subset \mathbb{CP}^2$, which is a union of the complex line $\mathbb{CP}^1$ with a finite number of symplectic lines, that is, symplectic $2$-spheres, each of which is homologous to $\mathbb{CP}^1 \subset \mathbb{CP}^2$. 
We call an intersection point $p$ of $S$ a \emph{multi-intersection point} if at least three symplectic lines pass through $p$. We denote the number of multi-intersection points in a symplectic line arrangement $S$ by $N_S$. Note that we blow up all the intersection points on the symplectic lines in $S$ to obtain an embedding $K$ in $M$, because each symplectic line becomes an arm in $K$. Therefore, all intersection points of symplectic lines in $S$ correspond to an exceptional $2$-sphere whose homology class appears at the first component of the corresponding arms in $K$, implying that the homological embedding of $K$ in $M$ determines the intersection data of $S$. 

Now, we provide a necessary condition for $W$ to be obtained by a sequence of rational blowdowns. Assume that a minimal symplectic filling $W$ of $Y$ is obtained from another symplectic filling $W'$ by rationally blowing down a negative definite star-shaped plumbing graph $G$ which is symplectically embedded in $W'$.
If $G$ is `nicely' embedded in $W'$, we can track the homological data of $K$ after surgery. Furthermore, we can describe a symplectic line arrangement $S$ corresponding to $W$ in terms of a symplectic line arrangement $S'$ corresponding to $W'$. In particular, we claim that the difference between the numbers $N_{S}$ and $N_{S'}$ of multi-intersection points is at most one, which is a key ingredient for getting the following main result.


\begin{thm}
\label{thm1}
Suppose a Seifert $3$-manifold $Y(-b; (\alpha_1, \beta_1), (\alpha_2, \beta_2),\ldots, (\alpha_n, \beta_n))$ satisfies $b\geq n+1$. If a minimal symplectic filling $W$ of $Y$ with a canonical contact structure is obtained from the minimal resolution of the corresponding weighted homogeneous surface singularity by a sequence of rational blowdowns, then the number $N_S$ of multi-intersection points in a symplectic line arrangement $S$ corresponding to $W$ is at most one.
\end{thm}

Furthermore, if we restrict to the case $b\geq n+2$, the condition $N_S\leq 1$ in Theorem ~\ref{thm1} is also a sufficient condition for a minimal symplectic filling to be obtained via rational blowdown surgeries.

\begin{thm}
\label{thm2}
For a Seifert $3$-manifold $Y(-b; (\alpha_1, \beta_1), (\alpha_2, \beta_2),\ldots, (\alpha_n, \beta_n))$ with $b\geq n+2$, any minimal symplectic filling $W$ of $Y$ with $N_S\leq 1$ is obtained by a sequence of rational blowdowns from the minimal resolution of the corresponding weighted homogeneous surface singularity.
\end{thm}

 A strategy for proving Theorem~\ref{thm2} is similar to that for proving Theorem 1.1 in ~\cite{CP2}. We divide all possible minimal symplectic fillings into certain types and then we show that such a sequence of rational blowdowns from the minimal resolution for each type exists by using lemmas proved in Section 4 ~\cite{CP2}.

Note that, if we further restrict to the case $b\geq n+3$, it is easy to check that every possible symplectic line arrangement satisfies the condition $N_S\leq 1$ (see Lemma ~\ref{linelem}). Hence we derive the following result from Theorem~\ref{thm2}.

\begin{cor}
For a Seifert $3$-manifold $Y(-b; (\alpha_1, \beta_1), (\alpha_2, \beta_2),\ldots, (\alpha_n, \beta_n))$ with $b\geq n+3$, every minimal symplectic filling of $Y$ is obtained by a sequence of rational blowdowns from the minimal resolution of the corresponding weighted homogeneous surface singularity.
\end{cor}

\begin{remark}
A family of minimal symplectic fillings of Seifert $3$-manifolds that cannot be obtained by a sequence of rational blowdowns were first provided by L.~Starkston in ~\cite{Sta2}.
Starkston's examples have $N_S=2$ with $b=n+2$. Hence, we easily recover Starkston's result using Theorem~\ref{thm1} above.
\end{remark}

\begin{remark}
In Theorem \ref{thm1} and Theorem \ref{thm2} above, the term `\emph{a sequence of rational blowdowns}' means that there is a sequence $W_i$ $(0\leq i \leq n_0)$ of minimal symplectic fillings of $(Y,\xi_{can})$ starting from the minimal resolution $W_{n_0}$ with $W_{0}\cong W$, and each $W_{i-1}$ is obtained from $W_i$ by rationally blowing down $G_i$, which is a negative definite star-shaped plumbing of $2$-spheres symplectically embedded in $W_i$. Hence, in general, we cannot find the plumbing graph $G_i$ in the dual resolution graph of the minimal resolution. However, if we allow blowing-ups from the resolution graph of the minimal resolution as in the quotient surface singularity cases ~\cite{CPS}, the plumbing graph $G_i \subset W_i$ can be found in most cases.     
\end{remark}

Finally, as an application of the main results above, we obtain a relation between minimal symplectic fillings of $Y(-b; (\alpha_1, \beta_1), (\alpha_2, \beta_2),\ldots, (\alpha_n, \beta_n))$ with $b\geq n+2$ and Milnor fibers of a weighted homogeneous surface singularity $(X,0)$ corresponding to $Y$ in the Appendix.

We call a proper flat map $\pi \colon \mathcal{X} \rightarrow \Delta$ with $\Delta=\{ t \in \mathbb{C} : |t| < \epsilon \}$ a \emph{smoothing} of $(X,0)$ if it satisfies $\pi^{-1}(0) = X$ and $\pi^{-1}(t)$ is smooth for all $t \neq 0$. The \emph{Milnor fiber} $M$ of a smoothing $\pi$ of $(X,0)$ is defined to be an intersection of a general fiber $\pi^{-1}(t)$ ($0 < t < \epsilon$) with a small closed ball centered at the origin. It is known that the Milnor fiber $M$ is a compact 4-manifold with the link $L$, which is diffeomorphic to $Y$, as its boundary and the diffeomorphism type of $M$ depends only on the smoothing $\pi$. 
Furthermore, $M$ has a natural symplectic structure, so that  it provides an example of minimal symplectic fillings of $(Y,\xi_{can})$. Hence, it is natural to ask the following question: ``For a given minimal symplectic $W$ of $Y$, is there a Milnor fiber $M$ of $(X,0)$ diffeomorphic to $W$?'' The answer is `no' in general because there is an infinite family of minimal symplectic fillings of a Seifert $3$-manifolds $Y(-b; (\alpha_1, \beta_1), (\alpha_2, \beta_2),\ldots, (\alpha_n, \beta_n))$ that cannot be diffeomorphic to any Milnor fibers~\cite{PS}. Note that all those examples satisfy $b=n+1$. Here we give a sufficient condition for an affirmative answer to the question.  More precisely, if a minimal symplectic filling $W$ of $Y$ satisfies $N_S\leq 1$, then there is a certain partial resolution $f: (Z,E)\rightarrow (X,0)$ (so-called $P$-resolution) such that the Milnor fiber of a smoothing of $Z$ is diffeomorphic to a given $W$. Hence, we get a deep relation between symplectic fillings and Milnor fibers for some Seifert 3-manifolds.

\begin{thm}
For a Seifert $3$-manifold $Y(-b; (\alpha_1, \beta_1), (\alpha_2, \beta_2),\ldots, (\alpha_n, \beta_n))$ with $b\geq n+2$, any minimal symplectic filling $W$ of $Y$ with $N_S\leq 1$ is realized as a Milnor fiber of some $P$-resolution of $(X,0)$.
\label{a}
\end{thm}

Furthermore, if $b\geq n+3$, every minimal symplectic filling satisfies automatically $N_S\leq 1$. Hence we also conclude that

\begin{cor}
For a Seifert $3$-manifold $Y(-b; (\alpha_1, \beta_1), (\alpha_2, \beta_2),\ldots, (\alpha_n, \beta_n))$ with $b\geq n+3$,
every minimal symplectic filling $W$ of $Y$ is realized as a Milnor fiber of some $P$-resolution of $(X,0)$.
\end{cor}

\subsection*{Acknowledgements}
The authors thank all members of the 4-manifold topology group at SNU for their helpful comments during the work. Jongil Park was supported by the National Research Foundation of Korea (NRF) grant funded by the Korean government (No.2020R1A5A1016126 and No.2021R1A2C1095776). 
He also holds a joint appointment at the Research Institute of Mathematics, SNU.

\section{Preliminaries}
\label{pre}

\subsection{Weighted homogeneous surface singularities and Seifert $3$-manifolds}
We briefly recall the relation between a Seifert $3$-manifold $Y$ and link $L$ of a weighted homogeneous surface singularity $(X,0)$.
We say that a normal surface singularity $(X,0)$ is a weighted homogeneous surface singularity if $(X,0)$ is given by zero loci of weighted homogeneous polynomials of the same type. Note that a polynomial $f(z_0,\dots, z_m)$ is called \emph{weighted homogeneous} if there exist nonzero integers $(q_0,\dots,q_m)$ and a positive integer $d$ that satisfy
$$f(t^{q_0}z_0,\dots t^{q_m}z_m)=t^df(z_0,\dots, z_m).$$
Then, there is a natural $\mathbb{C}^*$-action on $(X,0)$ given by 
$$t\cdot(z_0,\dots, z_m)=(t^{q_0}z_0,\dots t^{q_m}z_m),$$
which induces a fixed point-free $S^1\subset \mathbb{C}^*$ action on link $L:=X\cap \partial B$ of the singularity, where $B$ is a small ball centered at the origin. Hence, link $L$ is a Seifert fibered $3$-manifold over a genus $g$ Riemann surface. In this paper, we only consider a Seifert fibered $3$-manifold over the $2$-sphere, which is denoted by $Y(-b; (\alpha_1, \beta_1), (\alpha_2, \beta_2), \dots, (\alpha_n, \beta_n))$ for some integers $b, \alpha_i$ and $\beta_i$ with $0<\beta_i<\alpha_i$ and $(\alpha_i, \beta_i)=1$. 
Note that $n$ is the number of singular fibers, and there is an associated star-shaped plumbing graph $\Gamma$: the central vertex has genus $0$ and weight (equivalently, degree) $-b$, and each vertex in $n$ arms has genus $0$ and weight $-b_{ij}$ uniquely determined by the continued fraction 
$$\frac{\alpha_i}{\beta_i}=[ b_{i1}, b_{i2}, \dots, b_{ir_i} ]=b_{i1}-\displaystyle {\frac{1}{b_{i2}-\displaystyle\frac{1}{\cdots-\displaystyle\frac{1}{b_{ir_i}}}}}$$ 
with $b_{ij}\geq 2$. 
From P.~Orlik and P.~Wagreich~\cite{OW}, it is well known that the plumbing graph $\Gamma$ is a dual graph of the minimal resolution of $(X, 0)$. 
Moreover, if the intersection matrix of $\Gamma$ is negative definite, there is a weighted homogeneous surface singularity whose dual graph of the minimal resolution is $\Gamma$ ~\cite{Pin}. Furthermore, if a Seifert $3$-manifold $Y$ can be viewed as the link $L$ of a weighted homogeneous surface singularity, 
there exists a canonical contact structure $\xi_{\text{can}}$, called the \emph{Milnor fillable} contact structure, on $Y$ given by complex tangencies $TL\cap JTL$ that is known to be unique up to contactomorphism~\cite{CNPo}.

\subsection{Minimal symplectic fillings of Seifert $3$-manifolds}
In this subsection, we briefly review well-known facts regarding the minimal symplectic fillings of a Seifert $3$-manifold $Y$ with a canonical contact structure $\xi_{\text{can}}$.


 
As mentioned in the Introduction, there is a star-shaped plumbing graph $\Gamma$ associated to $Y$ (refer to Figure~\ref{Seifert}).
While $b\geq(n+1)$, we can always choose a concave cap $K$ of $(Y,\xi_{\text{can}})$ as shown in Figure~\ref{X}. For a minimal symplectic filling $W$ of $(Y,\xi_{\text{can}})$, we obtain a closed symplectic $4$-manifold $M=W\cup K$ by gluing $K$ along $Y$ to $W$. Then, the existence of $(+1)$ $2$-sphere in $K$ implies that $M$ is a rational symplectic $4$-manifold and, after a finite number of blowing-downs, $M$ becomes $\mathbb{CP}^2$ so that the $(+1)$ $2$-sphere in $K$ remains a complex line $\mathbb{CP}^1 \subset \mathbb{CP}^2$ (see Mcduff~\cite{McD} for details). The image of $K$ under the blowing-downs is called a \emph{symplectic line arrangement} $S$ consisting of complex line $\mathbb{CP}^1$ together with finite number of symplectic lines, in fact symplectic $2$-spheres, each of which is homologous to $\mathbb{CP}^1 \subset \mathbb{CP}^2$~\cite{Sta1}. 
Therefore, a minimal symplectic filling $W$ is completely determined by the homological embedding of $K$ in $M\cong \mathbb{CP}^2\sharp N\overline{\mathbb{CP}^2}$ and the isotopy type of $S$ in $\mathbb{CP}^2$. 
Note that the second homology group of $M$ is generated by $\{l,e_1,\dots, e_N\}$, where $l$ is a homology class of $\mathbb{CP}^1\subset \mathbb{CP}^2$ and $\{e_i\}$ are homology classes of exceptional $2$-spheres. Therefore, the homology class of each irreducible component of $K$ can be expressed in terms of this basis, which we call the \emph{homological data of $K$} for $W$. In Theorem~\ref{thm2}, we claim that, if the number $N_S$  of multi-intersection points of a symplectic line arrangement $S$ corresponding to $W$ is at most one,  the minimal symplectic filling $W$ of $(X,0)$ is obtained from the minimal resolution of $X$ by a sequence of rational blowdowns. Because the isotopy type of a symplectic line arrangement $S$ with a fixed intersection data is known to be unique if $N_S\leq 1$ (Proposition 4.2 in \cite{Sta2}), the minimal symplectic filling $W$ in Theorem~\ref{thm2} is determined uniquely by the homological data of $K$ for $W$.

Moreover, the combinatorial data of a symplectic line arrangement  $S$ can be described by a configuration of strands, as in Figure~\ref{line}. Each strand represents a symplectic $2$-sphere, and an intersection between strands represents a transversely geometric intersection between the $2$-spheres.
Hence, starting from a configuration of strands representing $S$, we can draw a configuration $C$ of strands containing $K$ using the homological data of $K$ for $W$. If there are no strands with degree less than or equal to $-2$ in $C$ except for the irreducible components of $K$, we call $C$ the \emph{curve configuration} for $W$, which is unique up to equivalence (Proposition 3.1 in~\cite{CP2}).

\smallskip

{\it{Terminology}}: We often use a terminology \emph{configuration of strands} when we deal with an intermediate configuration between a symplectic line arrangement and a curve configuration, or a configuration containing $K$ but there are strands with degree less than or equal to $-2$ other than irreducible components of $K$.

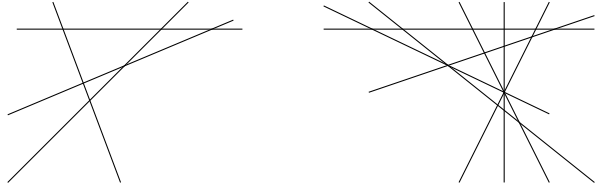
\begin{figure}[h]
\begin{tikzpicture}[scale=1.2]
\begin{scope}
\draw (0.1,0.2)--(2.6,0.2);
\draw (0.5,0.5)--(1.25,-1.5);
\draw (2,0.5)--(0,-1.5);
\draw (2.5,0.3)--(-0,-.75);
\end{scope}
\begin{scope}[shift={(4.5,0)}]
\draw (-1,0.2)--(2,0.2);
\draw (1,0.5)--(1,-1.5);
\draw (1.5,0.5)--(0.5,-1.5);
\draw (0.5,0.5)--(1.5,-1.5);
\draw (-.5,0.5)--(2,-1.5);
\draw (2,0.35)--(-.5,-.5);

\draw (1.5,-1938*1.5/4047-171/8094)--(-1,1938/4047-171/8094);

\end{scope}

\end{tikzpicture}
\caption{Examples of symplectic line arrangements}
\label{line}
\end{figure}

\subsection{Pseudo-holomorphic curves in rational symplectic $4$-manifolds}
Assume that a minimal symplectic filling $W$ of $Y$ is obtained from another minimal symplectic filling $W'$ by rationally blowing down a negative definite star-shaped plumbing graph $G$ that is symplectically embedded in $W'$.
To observe the effect of rationally blowing down $G\subset W'$ on a symplectic line arrangement, we first need to know how $G$ is symplectically embedded in $W'$. For this, we introduce several lemmas to analyze a symplectic embedding $G$ in $M=W'\cup K$. We assume that all irreducible components of $K$ and $G$ are $J$-holomorphic for a suitable tamed $J$. The following are some basic lemmas regarding $J$-holomorphic curves in $M$ obtained in ~\cite{BOn}.

\begin{lem}[\cite{BOn}]
\label{lem1} 
Let $L, C_1,\dots,C_k$ be a collection of symplectic $2$-spheres in a closed symplectic $4$-manifold $M$ with $L\cdot L=1$, $C_i\cdot C_i \leq 0$. Suppose that $J$ is a tame almost complex structure for which $L, C_1,\dots,C_k$ are $J$-holomorphic. Then there exists at least one $J$-holomorphic $(-1)$ curve in $M \setminus L$.
\end{lem}

\begin{lem}[\cite{BOn}]
\label{lem2}
Let $M$ be a closed symplectic $4$-manifold and let $L$ be a symplectically embedded $2$-sphere of self-intersection number $1$. Then, no symplectically embedded $2$-sphere of nonnegative self-intersection number is contained in $M\setminus L$. Pseudo-holomorphic $(-1)$ curves in $M\setminus L$ are mutually disjoint.
\end{lem}

\begin{lem}[\cite{BOn}]
\label{lem3}
Let $M$ be a closed symplectic $4$-manifold and let $L$ be a symplectically embedded $2$-sphere of self-intersection number $1$. Then, any irreducible singular or higher-genus pseudo-holomorphic curve $C$ in $M$ satisfies $C\cdot L \geq 3$. In particular, neither an irreducible singular nor a higher-genus pseudo-holomorphic curve is contained in $M\setminus L$.
\end{lem}

From Lemma~\ref{lem1}, we obtain a sequence of rational symplectic $4$-manifolds $M_j$ $(0\leq j \leq N)$ with $M_0\cong \mathbb{CP}^2$ and $M_N=M \cong  \mathbb{CP}^2 \sharp N\overline{\mathbb{CP}^2}$ such that $M_{j}$ is obtained by blowing down the $J_{j+1}$-holomorphic $(-1)$ curve $e_{j+1}$ from $M_{j+1}$ for a tamed $J_{j+1}$. 
Note that for a $J$-holomorphic $(-1)$ curve $e$ and an irreducible component $C$ of $G$ and $K$ in $M$, either $C$ is disjoint from $e$ or $C$ intersects transversally once with $e$ due to Lemma~\ref{lem3}. Hence, the image of $C$ under the blowing-downs in $M_j$ is a non-singular $J_j$-holomorphic curve. In particular, the self-intersection number of $C$ increases to $-1$. Therefore, $C$ eventually becomes the $J_j$-holomorphic curve $e_j$ under the blowing-downs unless $C$ is $C^0$ or $C^i_1$ for some $i$, which becomes an irreducible component of a symplectic line arrangement in $M_0\cong\mathbb{CP}^2$. Here $C^i_j$ denotes the $j^{\text{th}}$  irreducible component of the $i^{\text{th}}$ arm of $K$, and $C^0$ denotes the central $2$-sphere.

\begin{lem}
If there is a triple intersection between the images of the irreducible components of $K$ and $G$ during the blowing-downs, then they are the images of $C^{i_1}_1$, $C^{i_2}_1$ and $C^{i_3}_1$ for some $i_1$, $i_2$ and $i_3$ under the blowing-downs. 
\label{triple}
\end{lem}

\begin{figure}[h]
\begin{tikzpicture}[scale=0.9]
\draw (0,1.3)--(0,-.25);
\draw (0.2,.2)--(-1,-1);
\draw (-0.2,.2)--(1,-1);
\filldraw (0,0) circle (1.2pt);
\end{tikzpicture}
\caption{Pseudo-holomorphic curves with a triple intersection}
\end{figure}
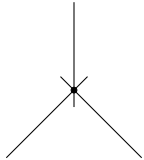

\begin{proof}
If one of the three pseudo-holomorphic curves does not come from $C^i_1$ of $K$, then the curve eventually becomes a $J_j$-holomorphic $(-1)$ curve $e_j$; thus, we have two pseudo-holomorphic curves with tangential intersection by blowing down $e_j$. If the other two pseudo-holomorphic curves come from the first components of $K$, then we have two symplectic lines in $S$ that do not intersect transversally, contradicting the definition of a symplectic line arrangement. Otherwise, we eventually have a singular curve intersecting the complex line $\mathbb{CP}^1$ at most once, which contradicts Lemma~\ref{lem3}.
\end{proof}

\section{Proof of Theorem~\ref{thm1}}
\label{sec-3}
To prove Theorem~\ref{thm1}, we first analyze the effect on symplectic line arrangements under a single rational blowdown surgery. In particular, we investigate the difference between two symplectic line arrangements $S$ and $S'$ corresponding to two minimal symplectic fillings $W$ and $W'$, respectively, where $W$ is obtained from $W'$ by rationally blowing down a negative definite star-shaped plumbing graph $G$ symplectically embedded in $W'$. 

First, we note how $J$-holomorphic curves intersect $K$ and $G$ in $M=W'\cup K$ using lemmas in Section ~\ref{pre}. 
Let $D^i_j$ be the $j^{\text{th}}$ irreducible component of the $i^{\text{th}}$ arm of $G$ and $D^0$ be the central $2$-sphere of $G$.

\begin{prop}
\label{prop1}
For the last component $D^i_{a_i}$ of each $i^{\text{th}}$ arm in $G$, there is a $J$-holomorphic  $(-1)$ curve $e_i$ and a linear chain $L_i$ (possibly empty) of the $J$-holomorphic curves in $M$ such that $D^i_{a_i}$ intersects with one end of $L_i$ and  $e_i$ connects with the other end of $L_i$ and an irreducible component of $K$. Furthermore, we eliminate the $i^{\text{th}}$ arm of $G$ by blowing down $(-1)$ curves consecutively starting from $e_i$.
\end{prop}

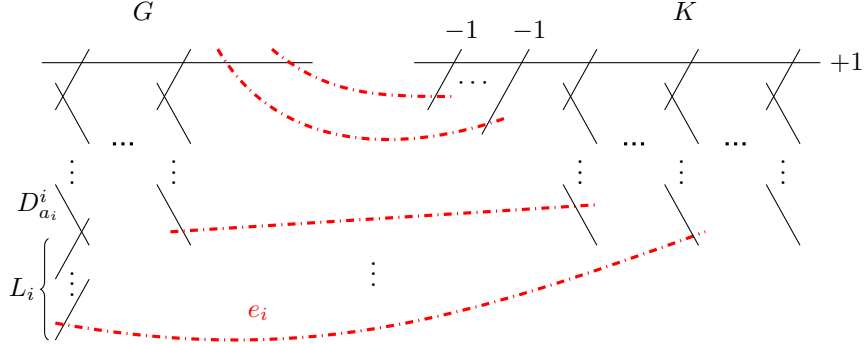
\begin{figure}[h]
\begin{tikzpicture}[scale=0.9]
\begin{scope}
\node[above] at (2.3, 0.5) {$G$};

\draw(0.8,0)--(4.8,0);
\draw (1.5,0.2)--(1,-0.7);
\draw (3.,0.2)--(2.5,-0.7);

\draw[dotted,very thick](2.15-.3,-1.2)--(2.5-.3,-1.2);

\draw (1,-0.3)--(1.5,-1.2);
\draw (2.5,-0.3)--(3.,-1.2);

\node at (1.25, -1.5) {$\vdots$};
\node at (2.75, -1.5) {$\vdots$};

\draw (1,-1.8)--(1.5,-2.7);
\draw (1.5,-2.3)--(1,-3.2);
\node at (1.25, -3.15) {$\vdots$};
\draw (1.5,-3.2)--(1,-4.1);
\draw [decorate,decoration={brace,mirror,amplitude=3pt},xshift=0pt,yshift=0pt]
	(0.9,-2.6) -- (0.9,-4.1) node [black,midway,xshift=-10pt] 
	{$L_i$};

\node at (0.75, -2.1) {$D^i_{a_i}$};
\node[red] at (4, -3.7) {$e_i$};

\draw (2.5,-1.8)--(3,-2.7);
\draw[red,dashdotted,very thick](4.2,0.2) to[out=-40,in=180] (6.9,-.5);
\draw[red,dashdotted,very thick](3.4,0.2) to[out=-60,in=200] (7.7,-.8);

\draw[red,dashdotted,very thick](2.7,-2.5)--(9,-2.1);
\draw[red,dashdotted,very thick](1.0,-3.85) to[out=-10,in=200](10.6,-2.5);


\node at(5.7,-3){$\vdots$};
\end{scope}

\begin{scope}[shift={(8.5,0)}]
\node[above] at (1.8,0.5){$K$};
\draw(-2.2,0)--(3.8,0);
\draw (0.5,0.2)--(0,-0.7);
\draw (.5+1.5,0.2)--(0+1.5,-0.7);
\draw (3.5,0.2)--(3,-0.7);

\draw (-.5,0.2)--(-1-.2,-0.7-.2*9/5);

\draw (-1.5,0.2)--(-2,-0.7);
\node at (-1.3,-0.3) {$\cdots$};
\draw[dotted,very thick](2.15+.25-1.5,-1.2)--(2.5+.25-1.5,-1.2);
\draw[dotted,very thick](2.15+.25,-1.2)--(2.5+.25,-1.2);

\draw (0,-0.3)--(0.5,-1.2);
\draw (0+1.5,-0.3)--(0.5+1.5,-1.2);

\draw (3,-0.3)--(3.5,-1.2);

\node at (0.25, -1.5) {$\vdots$};
\node at (0.25+1.5, -1.5) {$\vdots$};

\node at (3.25, -1.5) {$\vdots$};

\draw (0,-1.8)--(0.5,-2.7);
\draw (0+1.5,-1.8)--(0.5+1.5,-2.7);

\draw (3,-1.8)--(3.5,-2.7);

\node[right] at (3.8,0) {$+1$};
\node[above] at (-.5, 0.2) {$-1$};
\node[above] at (-1.5, 0.2) {$-1$};



\end{scope}

\end{tikzpicture}

\caption{$J$-holomorphic curves in $M$ intersecting $K$ and $G$}
\end{figure}

\begin{proof}
Note that every $J$-holomorphic $(-1)$ curve in $M=W'\cup K$ intersects some irreducible components of $K$ because $W'$ is a minimal symplectic filling. 
%
%
%
Let $D$ be an irreducible component of $G$ that first becomes a pseudo-holomorphic $(-1)$ curve during the blowing-downs from $M=W'\cup K$ to $\mathbb{CP}^2$. Then, there should exist a linear chain of $J$-holomorphic curves $D=D_0,\dots, D_k$ in $M$ such that the last component $D_k$ is a $(-1)$ curve, and the degree of $D_i$ $(1\leq i\leq k-1)$ is less than that of $D_k$ because we cannot increase the degree of $D$ without such a linear chain.  
Hence, we find a linear chain $L$ of $J$-holomorphic curves consisting of $D_1, \dots, D_{k-1}$ with a $J$-holomorphic $(-1)$ curve $e=D_k$ such that $D$ intersects with one end of $L$ and $e$ intersects with the other end of $L$.
Note that $e$ intersects only one irreducible component of $K$ due to Lemma~\ref{triple}. Furthermore, $D$ must be the last component $D^i_{a_i}$ of some $i^{\text{th}}$ arm of $G$. Otherwise, we would have a triple intersection consisting of the images of adjacent components of $D$ and an irreducible component of $K$ intersecting $e$, which is a contradiction.

Suppose there is another linear chain $L'$ and a $(-1)$ curve $e'$ intersecting $D$ as $L$ and $e$. Subsequently, an adjacent component of $D$ with irreducible components of $K$ intersecting $e$ and $e'$ would result in a triple intersection that contradicts Lemma~\ref{triple}. 
Therefore, starting from blowing down $e$, $D$ becomes a $(-1)$ curve under the blowing-downs along $(-1)$ curves coming from a linear chain of $J$-holomorphic curves consisting of $L$, $e$ and some irreducible components of $K$ connected to $D$ via $L$ and $e$. 
Let $G'$ be the image of $G$ under blowing-downs of the $(-1)$ curves above with the $(-1)$ curve coming from $D$. 
Then, $G'$ is still a star-shaped plumbing graph that has the same number of arms with $G$, and the number of irreducible components of $i^{\text{th}}$ arm in  $G'$ is less than that of $G$ by one. Then, using the same argument as before, we see that the last component of the $i^{\text{th}}$ arm in $G'$ is the first irreducible component becoming a $(-1)$ curve among the irreducible components of the $i^{\text{th}}$ arm in $G'$. 
We repeat the same process until all irreducible components of the $i^{\text{th}}$ arm in $G$ disappear under blowing-downs. Furthermore, by performing the same process for each arm in $G$, we conclude that $G$ eventually reduces to a single pseudo-holomorphic rational curve, which is the image of $D^0$, under the blowing-downs.  
\end{proof}

Unlike each arm of $G$, there may be several linear chains of $J$-holomorphic curves in $M$ intersecting $D^0$.
The next proposition shows how $G$ is obtained under the blowing-ups from $\mathbb{CP}^2$ to $M=W'\cup K$.
 
\begin{prop}
\label{prop2}
Let $T'$ be a subset of a symplectic line arrangement $S'$ consisting of the image of arms in $K$ connected to $G$ via $J$-holomorphic curves in $M$ under the blowing-downs from $M$ to $\mathbb{CP}^2$. Then, $T'$ has a unique intersection point, and $G$ is obtained by a sequence of blowing-ups from this point.
\end{prop}

\begin{proof}
We arrange a sequence of blowing-downs from $M=W'\cup K$ to $\mathbb{CP}^2$ into two steps: first blow down all $(-1)$ curves that only intersect $K$ and the image of $K$, and then blow down all $(-1)$ curves intersecting $G$ and the image of $G$ to obtain the image $T'\subset S'$ of arms in $K$ connected to $G$ via $J$-holomorphic curves in $M$.
 
First, note that for each arm of $K$, there is at most one arm of $G$ connected to the arm of $K$ via $J$-holomorphic curves; otherwise, we have cycles of $J$-holomorphic curves, which contradicts Lemma~\ref{lem3}.
Now, by the first step of the blowing-downs, the linear chain $L_i$ with $(-1)$ curve $e_i$ in Proposition~\ref{prop1} reduces to a single $(-1)$ curve $e'_i$, and there may  be several $(-1)$ curves intersecting the central curve $D^0$ of $G$.
Then, when we blow down $e'_i$, one of the two curves intersecting $e'_i$ becomes a $(-1)$ curve. 
Because all the irreducible components of each arm in $G$ disappear from the last to the first component, $G$ reduces to a single pseudo-holomorphic curve, which is the image of $D^0$ by blowing down all $(-1)$ curves consecutively. We further blow down $(-1)$ curves so that $D^0$ eventually becomes a $(-1)$ curve $e$. 

Because of the aforementioned blowing-down process, $e$ intersects the image of arms in $K$ connected to $G$ via $J$-holomorphic curves in $M$. Moreover, $e$ corresponds to the last step in the sequence of blowing-downs from $M$ to $\mathbb{CP}^2$, which indicates that the image $T' \subset S'$ of the arms in $K$ connected to $G$ via $J$-holomorphic curves has a unique intersection point.

\begin{figure}[h]
\begin{tikzpicture}[scale=0.9]
\begin{scope}
\node[above] at (2.3, 0.5) {$G$};

\draw(0.8,0)--(4.8,0);
\draw (1.5,0.2)--(1,-0.7);
\draw (3.,0.2)--(2.5,-0.7);

\draw[dotted,very thick](2.15-.3,-1.2)--(2.5-.3,-1.2);

\draw (1,-0.3)--(1.5,-1.2);
\draw (2.5,-0.3)--(3.,-1.2);

\node at (1.25, -1.5) {$\vdots$};
\node at (2.75, -1.5) {$\vdots$};

\draw (1,-1.8)--(1.5,-2.7);
\draw (1.5,-2.3)--(1,-3.2);
\node at (1.25, -3.15) {$\vdots$};
\draw (1.5,-3.2)--(1,-4.1);
\draw [decorate,decoration={brace,mirror,amplitude=3pt},xshift=0pt,yshift=0pt]
	(0.9,-2.6) -- (0.9,-4.1) node [black,midway,xshift=-10pt] 
	{$L_i$};

\node at(5.7,-3){$\vdots$};
\draw (2.5,-1.8)--(3,-2.7);
\draw[red,dashdotted,very thick](4.2,0.2) to[out=-40,in=180] (6.9,-.5);
\draw[red,dashdotted,very thick](3.4,0.2) to[out=-60,in=200] (7.7,-.8);

\draw[red,dashdotted,very thick](2.7,-2.5)--(9,-2.1);
\draw[red,dashdotted,very thick](1.0,-3.85) to[out=-10,in=200](10.6,-2.5);


\begin{knot}[
	clip width=20,
	clip radius = 20pt,
	end tolerance = 10pt,
]


\draw[very thin,->](6.5,-5)--(6.5,-.8);
\draw[very thin,->](6.5,-5)--(7.3,-1.2);
\draw[very thin,->](6.5,-5)--(8.9,-2.8);
\draw[very thin,->](6.5,-5)--(10.4,-2.8);
\node at(8.7,-3.5){$\cdots$};

\node[below] at(6.5,-5){blow down to};
\end{knot}

\end{scope}

\begin{scope}[shift={(4.85,-7.2)},scale=1.5]
\draw (-1,0.2)--(2.75,0.2);
\draw [decorate,decoration={brace,amplitude=3pt},xshift=0pt,yshift=5pt]
	(0.5,0.5) -- (1.5,0.5) node [black,midway,yshift=10pt] 
	{$T'$};
\node[above] at (-1.1,0.45){$S'$};
\draw (1,0.5)--(1,-1.5);
\draw (1.75,0.5)--(0.25,-1.5);
\draw (0.25,0.5)--(1.75,-1.5);
\draw (-1,0.4)--(2,-1.5);
\draw (2.75,0.4)--(-.25,-1.5);
\node at (1.25,0.4) {$\cdots$};
\node at (2.15,0.4) {$\cdots$};
\node at (-0.25,0.4) {$\cdots$};
\filldraw (1,-.5) circle (1.35pt);
\node[left] at (0.95,-0.5){$p'$};
\end{scope}
\begin{scope}[shift={(8.5,0)}]
\node[above] at (1.8,0.5){$K$};
\draw(-2.2,0)--(3.8,0);
\draw (0.5,0.2)--(0,-0.7);
\draw (.5+1.5,0.2)--(0+1.5,-0.7);
\draw (-.5,0.2)--(-1-.2,-0.7-.2*9/5);

\draw (3.5,0.2)--(3,-0.7);
\draw (-.5,0.2)--(-1,-0.7);
\draw (-1.5,0.2)--(-2,-0.7);
\node at (-1.3,-0.3) {$\cdots$};
\draw[dotted,very thick](2.15+.25-1.5,-1.2)--(2.5+.25-1.5,-1.2);
\draw[dotted,very thick](2.15+.25,-1.2)--(2.5+.25,-1.2);

\draw (0,-0.3)--(0.5,-1.2);
\draw (0+1.5,-0.3)--(0.5+1.5,-1.2);

\draw (3,-0.3)--(3.5,-1.2);

\node at (0.25, -1.5) {$\vdots$};
\node at (0.25+1.5, -1.5) {$\vdots$};

\node at (3.25, -1.5) {$\vdots$};

\draw (0,-1.8)--(0.5,-2.7);
\draw (0+1.5,-1.8)--(0.5+1.5,-2.7);

\draw (3,-1.8)--(3.5,-2.7);

\node[right] at (3.8,0) {$+1$};
\node[above] at (-.5, 0.2) {$-1$};
\node[above] at (-1.5, 0.2) {$-1$};



\end{scope}

\end{tikzpicture}

\caption{The arms of $K$ connected to $G$ via $J-$holomorphic curves blow down to $T'\subset S'$}

\end{figure}
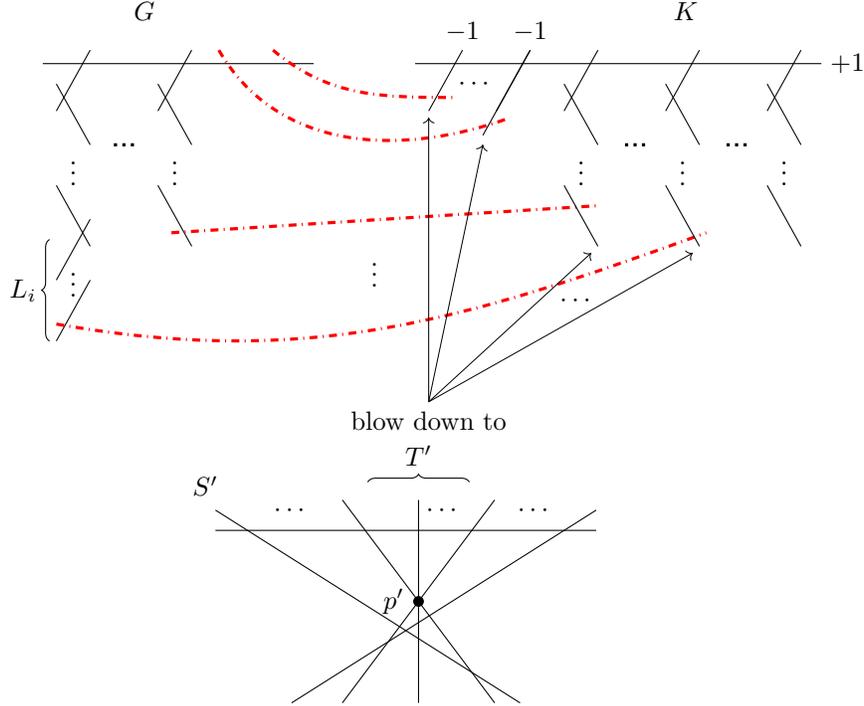
\end{proof}


Next, we investigate how $S'$ changes by rationally blowing down $G\subset W'$. 
Once we fix a sequence of blowing-downs along $J$-holomorphic $(-1)$ curves $E$ from $M=W'\cup K$ to $\mathbb{CP}^2$, there is a one-to-one correspondence between the set of intersection points in $S'$ and a subset of $E$ whose homology classes appear in more than one arm in $K$. 
Note that if we take another sequence of blowing-downs with the $J'$-holomorphic $(-1)$  curves $F$ from $M'$ to $\mathbb{CP}^2$, each homology class of $f_i\in F$ must be equal to that of some $e_j\in E$. Therefore, the intersection data of $S'$ are determined by a homological expression of $\{C^i_1\}\subset K$ in terms of a complex line $\mathbb{CP}^1$ and some $(-1)$ 2-spheres disjoint from the complex line. 

Now, we arrange a sequence of blowing-downs from $M=W'\cup K$ to $\mathbb{CP}^2$ into two steps, as in the proof of Proposition~\ref{prop2}. Let $E_G$ be a subset of $E$ whose homology classes appear in the homology classes of irreducible components in $G\subset M$. If $e\in E\setminus E_G$ represents an intersection point of $S'$, then $e$ also represents an intersection point of a symplectic line arrangement $S$ corresponding to $W$ because $e$ is a  $(-1)$ curve in $M \setminus G$. Furthermore, since $G$ is obtained by a sequence of blowing-ups from a unique intersection point of $T'\subset S'$, there is at most one $(-1)$ curve in $E_G$ that corresponds to an intersection point of $S'$. Then, we obtain the following relation between $N_S$ and $N_{S'}$ under rationally blowing down along $G$ in $W'$.


\begin{prop}
\label{mainthm1}
If a minimal symplectic filling $W$ is obtained from $W'$
by rationally blowing down along $G$, then $N_S=N_{S'}$ or $N_S=N_{S'}-1$, where $S$ and $S'$ are symplectic line arrangements corresponding to $W$ and $W'$, respectively.
\end{prop}

\begin{proof}
Let $K_{T'}\subset K$ be a subset of arms in $K$ whose image under the blowing-downs is $T'$ in Proposition~\ref{prop2}. 
The observations above show that the intersection data of $S$ are equal to that of $S'$ except for the intersection data in $T\subset S$, where $T$ is the image of $K_{T'}$ under a sequence of blowing-downs from $W\cup K$ to $\mathbb{CP}^2$. Hence, we only need to show that $T$ has at most one multi-intersection point.

 As we saw in Proposition~\ref{prop2}, $G$ is obtained from the exceptional curve $e$ by blowing up at the unique intersection point $p'$ of $T'$. Therefore, the number of arms in $G$ is less than or equal to the number of points in $e$ which we blow up to get the central curve $D^0$ of $G$.  Hence, the absolute value of the degree of $D^0$ is strictly larger than the number of arms in $G$, so that $G$ must be linear or $\Gamma_{p,q,r}$ in Figure~\ref{pqr} because of Stipsicz and Bhupal's classification result~\cite{BS}.
 
\begin{figure}[h]
\begin{tikzpicture}[scale=0.6]
\node[bullet] at (0,0){};
\node[bullet] at (1,0){};
\node[bullet] at (3,0){};
\node[bullet] at (4,0){};
\node[bullet] at (5,0){};
\node[bullet] at (7,0){};
\node[bullet] at (8,0){};

\node[bullet] at (4,-1){};
\node[bullet] at (4,-3){};
\node[bullet] at (4,-4){};

\node[below left] at (0,0){$-(p+3)$};
\node[above] at (1,0){$-2$};
\node[above] at (3,0){$-2$};
\node[above] at (4,0){$-4$};
\node[above] at (5,0){$-2$};
\node[above] at (7,0){$-2$};
\node[below right] at (8,0){$-(q+3)$};

\node[left] at (4,-1){$-2$};
\node[left] at (4,-3){$-2$};
\node[below] at (4,-4){$-(r+3)$};

\node at (2,0){$\cdots$};
\node at (6,0){$\cdots$};
\node at (4,-1.9){$\vdots$};

\draw (0,0)--(1.5,0);
\draw (2.5,0)--(5.5,0);
\draw (6.5,0)--(8,0);
\draw (4,0)--(4,-1.5);
\draw (4,-2.5)--(4,-4);

	\draw [thick,decorate,decoration={brace,mirror,amplitude=5pt},xshift=0pt,yshift=-7pt]
	(1,0) -- (3,0) node [black,midway,yshift=-11pt] 
	{$q$};

	\draw [thick,decorate,decoration={brace,mirror,amplitude=5pt},xshift=0pt,yshift=-7pt]
	(5,0) -- (7,0) node [black,midway,yshift=-11pt] 
	{$r$};

	\draw [thick,decorate,decoration={brace,amplitude=5pt},xshift=0pt,xshift=7pt]
	(4,-1) -- (4,-3) node [black,midway,xshift=11pt] 
	{$p$};

\end{tikzpicture}
\caption{Plumbing graph $\Gamma_{p,q,r}$}
\label{pqr}
\end{figure}


Recalling the blowing-down process from $G$ to a point in the proof of Proposition~\ref{prop2}, we can observe that the effect of each blowing-down is either 
increasing the degree of an irreducible component or decreasing the length of an arm. Conversely, under the blowing-ups from $p'$ to $G$, we obtain a star-shaped plumbing of the symplectic $2$-spheres $K_G$ consisting of the complex line in $S'$ and the image of $T'\subset S'$. In particular, the effect of each blowing-up is either to decrease the degree of an irreducible component or to increase the length of an arm. Furthermore, the complement of $G$ in the resulting rational symplectic $4$-manifold $\widetilde{M}$ is $K_G$, indicating that $K_G$ is a concave cap of $(\partial G, \xi_{\text{can}})$. As $G$ is either linear or $\Gamma_{p,q,r}$, $K_G$ is represented by Figure~\ref{KG}. The degrees of unlabeled strands in $(b)$ are all $(-2)$.

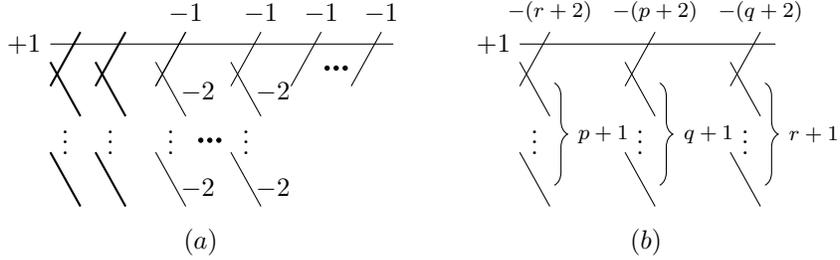
\begin{figure}[h]
\begin{tikzpicture}[scale=0.8]
\begin{scope}[shift={(3.2,0)}]

\node at (2.5, -3.3) {$(a)$};
\draw(0,0)--(5.7,0);
\draw[thick] (0.5,0.2)--(0,-0.7);
\draw[thick] (1.25,0.2)--(0.75,-0.7);

\draw (2.25,0.2)--(1.75,-0.7);
\node[above] at (2.25, .2) {$-1$};
\draw (3.25+.25,0.2)--(2.75+.25,-0.7);
\node[above] at (3.25+.25, .2) {$-1$};
\draw (2.25,0.2)--(1.75,-0.7);

\draw (4.25+.25,0.2)--(3.75+.25,-0.7);
\node[above] at (4.25+.25, .2) {$-1$};

\draw (5.25+.25,0.2)--(4.75+.25,-0.7);
\node[above] at (5.25+.25, .2) {$-1$};

\filldraw (5-.25,-.4) circle(1pt);
\filldraw (4.85-.25,-.4) circle(1pt);
\filldraw (5.15-.25,-.4) circle(1pt);

\node at (2.45, -0.8) {$-2$};
\node at (2.45, -2.4) {$-2$};

\node at (3.7, -0.8) {$-2$};
\node at (3.7, -2.4) {$-2$};

\filldraw (2.5,-1.6) circle(1pt);
\filldraw (2.65,-1.6) circle(1pt);
\filldraw (2.8,-1.6) circle(1pt);

\draw[thick] (0,-0.3)--(0.5,-1.2);
\draw[thick] (0.75,-0.3)--(1.25,-1.2);

\draw (1.75,-0.3)--(2.25,-1.2);
\draw (2.75+.25,-0.3)--(3.25+.25,-1.2);

\node at (0.25, -1.5) {$\vdots$};
\node at (1, -1.5) {$\vdots$};
\node at (2, -1.5) {$\vdots$};
\node at (3+.25, -1.5) {$\vdots$};

\draw[thick] (0,-1.8)--(0.5,-2.7);
\draw[thick] (.75,-1.8)--(1.25,-2.7);
\draw (1.75,-1.8)--(2.25,-2.7);
\draw (2.75+.25,-1.8)--(3.25+.25,-2.7);

\node[left] at (0,0) {$+1$};



\end{scope}
\begin{scope}[shift={(11,0)}]
\node at (2.1, -3.3) {$(b)$};
\draw(0,0)--(4+.25,0);
\draw (0.5,0.2)--(0,-0.7);
\node[above] at (0.5, 0.2) {\footnotesize$-(r+2)$};
\draw (2+.25,0.2)--(1.5+.25,-0.7);
\node[above] at (2+.25, 0.2) {\footnotesize$-(p+2)$};
\draw (3.5+.5,0.2)--(3+.5,-0.7);
\node[above] at (3.5+.5, 0.2) {\footnotesize$-(q+2)$};

	\draw [decorate,decoration={brace,amplitude=5pt},yshift=10pt,xshift=5pt]
	(0.4,-1) -- (0.4,-2.7) node [black,midway,xshift=18.5pt, yshift=0pt] 
	{\footnotesize$p+1$};

	\draw [decorate,decoration={brace,amplitude=5pt},yshift=10pt,xshift=5pt]
	(2.15,-1) -- (2.15,-2.7) node [black,midway,xshift=18.5pt, yshift=0pt] 
	{\footnotesize$q+1$};
	\draw [decorate,decoration={brace,amplitude=5pt},yshift=10pt,xshift=5pt]
	(3.65+.25,-1) -- (3.65+.25,-2.7) node [black,midway,xshift=18.5pt, yshift=0pt] 
	{\footnotesize$r+1$};

\draw (0,-0.3)--(0.5,-1.2);
\draw (1.5+.25,-0.3)--(2+.25,-1.2);
\draw (3+.5,-0.3)--(3.5+.5,-1.2);

\node at (0.25, -1.5) {$\vdots$};
\node at (2, -1.5) {$\vdots$};
\node at (3.25+.25+.25, -1.5) {$\vdots$};

\draw (0,-1.8)--(0.5,-2.7);
\draw (1.5+.25,-1.8)--(2+.25,-2.7);
\draw (3+.5,-1.8)--(3.5+.5,-2.7);

\node[left] at (0,0) {$+1$};


\end{scope}
\end{tikzpicture}
\caption{Concave cap $K_G$}
\label{KG}
\end{figure}

More specifically, $K_G$ is of the form (a) or (b) in Figure~\ref{KG} depending on whether $G$ is linear or $\Gamma_{p,q,r}$. Note that two unlabeled arms in (a) correspond to the two arms of a linear plumbing graph $G$ whereas the arms with only $(-1)$ or $(-2)$  strands in (a) contribute to the degree of $D^0$.
Let $K'$ be an image of $S'$ in $\widetilde{M}$ containing $K_G$ under the blowing-ups from $e$ to $G$. Then, we have a sequence of blowing-ups from $K'$ to $K$ in terms of $E\setminus E_G$, so that the homological data of $K$ in $M$ consist of the homological data of $K_G$ in $\widetilde{M}$ with the homological data from the blowing-ups from $K'$ to $K$. 
Similarly, the homological data of $K$ in $W\cup K$ consist of the homological data of $K_G$ in $(\widetilde{M}\setminus G) \cup B_G$ with the homological data from the blowing-ups from $K'$ to $K$ in terms of $E\setminus E_G$, where $B_G$ is a rational homology ball filling of $(\partial G, \xi_{\text{can}})$. 
As the arms in $K_G$ become $K_{T'}\subset K$, the intersection data of $T$ are determined by homological data of $K_G$ in $(\widetilde{M}\setminus G) \cup B_G$.  Specifically, the intersection data of $T$ are equal to those of a symplectic line arrangement corresponding to $B_G$ with respect to concave cap $K_G$.
Finally, since there are only two possible symplectic line arrangements in Figure~\ref{possibleline} for any minimal symplectic filling of $(\partial G, \xi_{\text{st}})$ with respect to $K_G$ due to the arms starting with $(-1)$ strands (refer to Proposition 3.2 in ~\cite{CP2} for details), the number of multi-intersection points in $T$ is at most one, as required. 

\begin{figure}[h]
\begin{tikzpicture}[scale=1.2]
\begin{scope}
\draw (0.1,0.2)--(2.6,0.2);
\draw (1,0.5)--(1,-1.5);
\draw (1.5,0.5)--(0.5,-1.5);
\draw (0.5,0.5)--(1.5,-1.5);
\draw (2,0.5)--(0,-1.5);
\draw (2.5,0.5)--(-.5,-1.5);
\node at (2.15,0.4) {$\cdots$};
\end{scope}
\begin{scope}[shift={(4.5,0)}]
\draw (-1,0.2)--(2,0.2);
\draw (1,0.5)--(1,-1.5);
\draw (1.5,0.5)--(0.5,-1.5);
\draw (0.5,0.5)--(1.5,-1.5);
\draw (-1,0.5)--(2,-1.5);
\node at (1.25,0.4) {$\cdots$};

\end{scope}

\end{tikzpicture}
\caption{Two possible symplectic line arrangements}
\label{possibleline}
\end{figure}
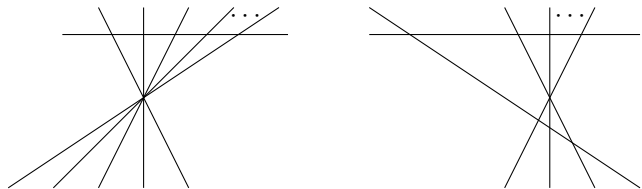

\end{proof}

\begin{proof}[Proof of Theorem~\ref{thm1}]
It follows from Proposition~\ref{mainthm1} and that the minimal resolution graph is obtained from the left-hand symplectic line arrangement in Figure~\ref{possibleline}, which has a unique multi-intersection point.
\end{proof}

\section{Proof of Theorem~\ref{thm2}}
\label{4}

In this section, we show that the converse of Theorem~\ref{thm1} also holds for a Seifert $3$-manifold $Y(-b; (\alpha_1, \beta_1), (\alpha_2, \beta_2),\ldots, (\alpha_n, \beta_n))$ with $b\geq n+2$. As mentioned in Section 2, a minimal symplectic filling $W$ with $N_S\leq 1$ is determined by the homological data of $K$ for $W$. Here $S$ is a symplectic line arrangement corresponding to $W$. Therefore, we need to analyze all possible curve configurations coming from $S$ with $N_S\leq 1$  to show Theorem~\ref{thm2}. The strategy for the proof is similar to the proof of Theorem 1.1 in ~\cite{CP2}. We divide all possible curve configurations into certain types and then show that there are sequences of rational blowdowns from the minimal resolution for each type using lemmas in Section 4, ~\cite{CP2}.
First, when $b\geq n+2$, we determine all possible symplectic line arrangements $S$ with $N_S\leq 1$.

\begin{lem}
Assume that $b\geq n+2$. If the number $N_S$ of multi-intersection points of a symplectic line arrangement $S$ is at most 1, then $S$ is one of the two symplectic line arrangements in Figure~\ref{possibleline}. 

\end{lem}

\begin{proof}

Since $b\geq n+2$, there is at least one arm in $K$ that consists of a single $(-1)$ $2$-sphere. 
Let $s\in S$ be an image of the $(-1)$ $2$-sphere under blowing-downs. Then, there are at most two intersection points on $s$ due to the degree. Because $N_S\leq 1$, there are only two possibilities: all symplectic lines in $S$ have a common intersection point or all symplectic lines have a common intersection except one symplectic line, which are left-hand and right-hand line arrangements in Figure~\ref{possibleline}, respectively.
\end{proof}

In fact, if $b\geq n+3$, then two symplectic line arrangements in Figure~\ref{possibleline} give all possible symplectic line arrangements  (cf. Lemma 2.5 in \cite{Sta1}). 

\begin{lem}
Assume that $b\geq n+3$. For minimal symplectic fillings of a Seifert fibered $3$-manifold $Y(-b; (\alpha_1, \beta_1), (\alpha_2, \beta_2),\dots, (\alpha_n, \beta_n))$, there are only two possible intersection relations of symplectic line arrangements listed in Figure~\ref{possibleline}. 
\label{linelem}
\end{lem}

\begin{proof}
Let $s\in S$ be an image of the $(-1)$ $2$-sphere under blowing-downs as before. If two intersection points on $s$ are all multi-intersection points, then degrees of the lines in $S$ except $s$ are strictly less than $-1$ after blowing-up all intersection points in $S$. This contradicts the fact that there are at least two arms in $K$ consisting of a single $(-1)$ $2$-sphere.
\end{proof}

When we attempt to obtain a curve configuration $C$ from a symplectic line arrangement $S$, we first blow up all intersection points between symplectic lines in $S$. Once we blow up an exceptional strand, we should blow up all intersection points of the strand except one to allow only strands with degree 
$\leq -2$, if each strand represents an irreducible component of $K$. Without loss of generality, we assume that the first $n$ arms become \emph{essential arms} in $K$ consisting of strands with degrees 
$\leq -2$. Based on this, we can divide all the possible curve configurations obtained from $S$ with $N_S\leq 1$ into the following three types:

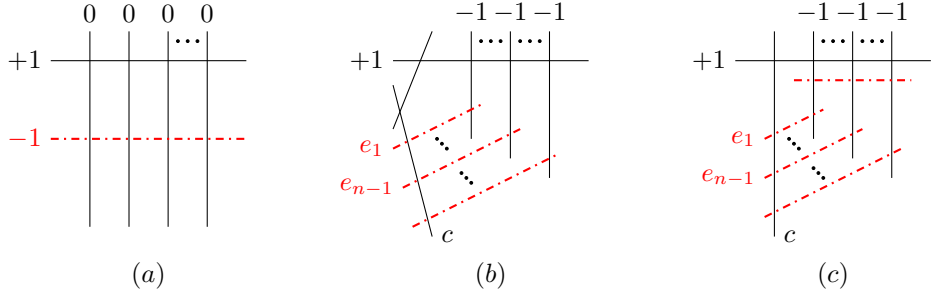
\begin{figure}[h]
\begin{center}
\begin{tikzpicture}[scale=1.3]
\begin{scope}[shift={(-3.5,0)}]

\draw (0,0.2)--(2,0.2);
\draw (0.4,0.5)--(0.4,-1.5);
\draw (0.8,0.5)--(0.8,-1.5);
\draw (1.2,0.5)--(1.2,-1.5);
\draw (1.6,0.5)--(1.6,-1.5);
\draw[red, thick, dashdotted] (0,-0.6)--(2,-0.6);
\node[left,red] at (0,-0.6) {$-1$};
\node[left] at (0,0.2) {$+1$};
\node at (1,-2) {$(a)$};

\filldraw (1.3,0.4) circle (.4pt);
\filldraw (1.4,0.4) circle (.4pt);
\filldraw (1.5,0.4) circle (.4pt);
\node[above] at (0.4,0.5) {$0$};
\node[above] at (0.8,0.5) {$0$};
\node[above] at (1.2,0.5) {$0$};
\node[above] at (1.6,0.5) {$0$};
\end{scope}

\begin{scope}
\node at (1,-2) {$(b)$};

\draw (0,0.2)--(2,0.2);
\draw (0.4,0.5)--(0,-0.5);
\draw (0,-0.05)--(0.4,-1.6);
\node[right] at (0.4,-1.6) {$c$};
\draw (0.8,0.5)--(0.8,-0.6);
\draw (1.2,0.5)--(1.2,-0.8);
\draw (1.6,0.5)--(1.6,-1);

\draw[red, thick, dashdotted](0,-0.7)--(0.9,-0.25);
\draw[red, thick, dashdotted](0.1,-1.1)--(1.3,-0.5);
\draw[red, thick, dashdotted](0.2,-1.5)--(1.7,-0.75);

\node[left, red] at (0,-0.7) {$e_1$};
\node[left, red] at (0.1,-1.1) {$e_{n-1}$};

\filldraw (1.3-.4,0.4) circle (.4pt);
\filldraw (1.4-.4,0.4) circle (.4pt);
\filldraw (1.5-.4,0.4) circle (.4pt);

\filldraw (1.3,0.4) circle (.4pt);
\filldraw (1.4,0.4) circle (.4pt);
\filldraw (1.5,0.4) circle (.4pt);
\node[above] at (0.8,0.5) {$-1$};
\node[above] at (1.2,0.5) {$-1$};
\node[above] at (1.6,0.5) {$-1$};
\node[left] at (0,0.2) {$+1$};

\filldraw (0.7-.25,-.95+.35) circle (.4pt);
\filldraw (0.75-.25,-1.+.35) circle (.4pt);
\filldraw (0.8-.25,-1.05+.35) circle (.4pt);

\filldraw (0.7,-.95) circle (.4pt);
\filldraw (0.75,-1.) circle (.4pt);
\filldraw (0.8,-1.05) circle (.4pt);

\end{scope}

\begin{scope}[shift={(3.5,0)}]

\filldraw (1.3-.4,0.4) circle (.4pt);
\filldraw (1.4-.4,0.4) circle (.4pt);
\filldraw (1.5-.4,0.4) circle (.4pt);

\filldraw (1.3,0.4) circle (.4pt);
\filldraw (1.4,0.4) circle (.4pt);
\filldraw (1.5,0.4) circle (.4pt);
\node at (1,-2) {$(c)$};

\draw (0,0.2)--(2,0.2);
\draw (0.4,0.5)--(0.4,-1.6);
\node[right] at (0.4,-1.6) {$c$};

\draw (0.8,0.5)--(0.8,-0.6);
\draw (1.2,0.5)--(1.2,-0.8);
\draw (1.6,0.5)--(1.6,-1);
\draw[red, thick, dashdotted] (0.6,0)--(1.8,0);
\draw[red, thick, dashdotted](0.3,-0.6)--(0.9,-0.3);
\draw[red, thick, dashdotted](0.3,-1)--(1.3,-0.5);
\draw[red, thick, dashdotted](0.3,-1.4)--(1.7,-0.7);
\node[left, red] at (0.3,-0.6) {$e_1$};
\node[left, red] at (0.3,-1) {$e_{n-1}$};

\filldraw (0.7-.15,-.95+.35) circle (.4pt);
\filldraw (0.75-.15,-1.+.35) circle (.4pt);
\filldraw (0.8-.15,-1.05+.35) circle (.4pt);

\filldraw (0.8,-.95+.05) circle (.4pt);
\filldraw (0.85,-1.+.05) circle (.4pt);
\filldraw (0.9,-1.05+.05) circle (.4pt);

\node[above] at (0.8,0.5) {$-1$};
\node[above] at (1.2,0.5) {$-1$};
\node[above] at (1.6,0.5) {$-1$};
\node[left] at (0,0.2) {$+1$};

\end{scope}
\end{tikzpicture}
\end{center}
\caption{Three configurations}
\label{startingposition}
\end{figure} 


\begin{itemize}
\item{Type A:}
Curve configurations obtained from $(a)$ in Figure~\ref{startingposition} without blowing up the exceptional strand.
\item{Type B:}
Curve configurations obtained from $(b)$ or $(c)$ in Figure~\ref{startingposition} by blowing up at most one $e_i$ $(1\leq i \leq n-1)$.
\item{Type C:}
Curve configurations obtained from $(b)$ or $(c)$ in Figure~\ref{startingposition} by blowing up at least two  $e_i$s $(1\leq i \leq n-1)$.
\end{itemize}

Note that $(a)$ and $(c)$ are obtained from left-hand and right-hand symplectic line arrangements in Figure~\ref{possibleline}, respectively, whereas $(b)$ is obtained from $(a)$ by blowing up the unique exceptional strand in $(a)$.
 
 We now recall several lemmas given in ~\cite{CP2} that are useful for finding a surgical description for a minimal symplectic filling of each type. We first recall the notion of \emph{standard blowing-ups}: for star-shaped $K'$ and $K$ of the same number of arms with central $(+1)$ vertex, we say $K'\leq K$ if $n'_i\leq n_i$ and $a'_{ij}\leq a_{ij}$ for any $i$ and $j$ except for $a'_{in'_i} < a_{in'_i}$ when $n'_i<n_i$, where $-a_{ij}$ $(1\leq j \leq n_i)$ and $-a'_{ij}$ $(1\leq j \leq n'_i)$ are the weights (equivalently, degrees) of the $j^{\text{th}}$-vertex in the $i^{\text{th}}$-arm of $K$ and $K'$, respectively.
 
 \begin{defn}
 Let $C'$ be a configuration of strands obtained from a symplectic line arrangement by blowing-ups which contains a star-shaped plumbing graph $K'$. If $K'\leq K$ and the degree of the strands except $K'$ is $-1$, we obtain a curve configuration $\widetilde{C'}$ from $C'$ by blowing up only at non-intersection points. That is, $\widetilde{C'}$ is obtained by blowing up the non-intersection points of  the last component of each arm of $K'$ consecutively to obtain $n_i$ components and then by blowing up the non-intersection points of each irreducible component to obtain the correct degree $a_{ij}$. In this case, we say that the curve configuration $\widetilde{C'}$ is obtained from $C'$ through \emph{standard blowing-ups}.
 \end{defn}

Next, we compare the standard blowing-up $\widetilde{C'}$ with a curve configuration $C$, which is obtained from $C'$ using non-standard blowing-ups.

 \begin{lem}[\cite{CP2}]
 Let $C$ be a curve configuration obtained from $C'$ by blowing-ups. If $C$ is differ from $\widetilde{C'}$ only in the components $C^i_j$ of the $i^{\text{th}}$-arm for \mbox{$n'_i\leq j \leq n_i$,} then a minimal symplectic filling $W$ corresponding to $C$ is obtained from a minimal symplectic filling $\widetilde{W}$ corresponding to $\widetilde{C'}$ by a sequence of rational blowdowns.
\label{fundamentallem1}
\end{lem}

In addition to the assumptions of Lemma~\ref{fundamentallem1}, we assume that there is a $(-1)$ strand intersecting both $C^{i}_{n_i'}$ and another irreducible component $C^k_{l}$ of $K'$ in $C'$.
Then, we have a slight modification of the Lemma~\ref{fundamentallem1} involving two arms of $K$, as follows.

\begin{lem}[\cite{CP2}]
Suppose that there is a $(-1)$ strand intersecting $C^{i}_{n_i'}$ and $C^{k}_{l}$ of $K'$ in $C'$ with $a'_{kl} < a_{kl}$. If the standard blowing-ups $\widetilde{C'}$ of $C'$ differs from $C$ only in $C^k_{l}$ and components $C^i_{j}$ for $n'_{i} \leq
j \leq n_i$, then a minimal symplectic filling $W$ corresponding to $C$ is obtained from a minimal symplectic filling $\widetilde{W}$ corresponding to $\widetilde{C'}$ by a sequence of rational blowdowns.
\label{fundamentallem2}
\end{lem}


Finally, if $K'$ is a concave cap for another Seifert $3$-manifold $Y'$, we have an explicit description of $\widetilde{W}$. For this purpose, let $X$ and $X'$ denote the corresponding weighted homogeneous surface singularities to $Y$ and $Y'$, respectively.

\begin{lem}[\cite{CP2}]
If $K'$ is a concave cap for a Seifert $3$-manifold $Y'$ such that $C'$ is a curve configuration,  there is a symplectic embedding of the minimal resolution of $X'$ to the minimal resolution of $X$ so that $\widetilde{W}$ is obtained from the minimal resolution of $X$ by replacing the minimal resolution of $X'$ with a minimal symplectic filling $W'$ of $Y'$ corresponding to $C'$.
\label{fundamentallem3}
\end{lem}

With these three fundamental lemmas, the proof of Theorem~\ref{thm2} is essentially identical to the proof of Theorem 1.1 for $b\geq 5$ case in \cite{CP2}, but we provide a detailed proof for completeness.

\subsection{Proof for type A}
Evidently all strands $K'$, except the exceptional strand in $(a)$ in Figure~\ref{startingposition}, satisfy $K'\leq K$. Hence, by repeatedly applying Lemma~\ref{fundamentallem1} to the arms of $K$, we show that any minimal symplectic filling $W$, whose corresponding curve configuration $C$ is of type A,  is obtained by a sequence of rational blowdowns from $\widetilde{W}$, where $\widetilde{W}$ is a minimal symplectic filling corresponding to the standard blowing-ups of $(a)$, which is known to be deformation equivalent to the minimal resolution of corresponding singularity. Actually, each minimal symplectic filling of type A is obtained by replacing each arm of $\Gamma$ with its minimal symplectic filling.

\subsection{Proof for type B}
Without loss of generality, we assume that the first and second arms of a configuration $(b)$ or $(c)$ in Figure~\ref{startingposition} become the first and second arms of $K$ in $C$, respectively,
and the proper transforms of $e_i$ $(2 \leq i \leq n-1)$ are not irreducible components of $K$. 
Since we do not blow up exceptional strands $e_i$s for $2 \leq i \leq n-1$, we can get the first and second arms of $K$, leaving the single $(-1)$ arms unchanged.
Hence, we arrange the order of blowing-ups from a configuration $(b)$ or $(c)$ in Figure~\ref{startingposition} to $C$ so that we have an intermediate configuration $C'$ of strands containing $K'\leq K$ as shown in Figure~\ref{subcap}.
Note that the degrees of strands in $C' \setminus K'$ are all $-1$ and the homological data of the first and second arms of $K'$ in $C'$ are equal to those of $K$ in $C$.

\begin{figure}[h]
\begin{center}
\begin{tikzpicture}[scale=1.2]
\draw(0,0)--(5,0);
\draw (0.5,0.2)--(0,-0.7);
\draw (1.5,0.2)--(1,-0.7);
\draw (2.5,0.2)--(2,-0.7);
\draw (3.5,0.2)--(3,-0.7);
\draw (4.5,0.2)--(4,-0.7);
\node at (4,0.1) {$\cdots$};

	\draw [decorate,decoration={brace,amplitude=5pt},xshift=0pt,yshift=3pt]
	(2.5,0.2) -- (4.5,0.2) node [black,midway,yshift=8pt] 
	{\footnotesize $b-3$};

\draw (0,-0.3)--(0.5,-1.2);
\draw (1,-0.3)--(1.5,-1.2);

\node at (0.25, -1.5) {$\vdots$};
\node at (1.25, -1.5) {$\vdots$};

\draw (0,-2)--(0.5,-2.9);
\draw (1,-2)--(1.5,-2.9);

\node[left] at (0,0) {$+1$};
\node[right] at (0.25, -0.25) {$-a_{11}$};
\node[right] at (1.25, -0.25) {$-a_{21}$};

\node[right] at (2.25, -0.25) {$-1$};
\node[right] at (3.25, -0.25) {$-1$};
\node[right] at (4.25, -0.25) {$-1$};

\node[right] at (0.25, -0.75) {$-a_{12}$};
\node[right] at (1.25, -0.75) {$-a_{22}$};

\node[right] at (0.25, -2.45) {$-a_{1n_1}$};
\node[right] at (1.25, -2.45) {$-a_{2n_2}$};

\end{tikzpicture}
\end{center}
\caption{Concave cap $K'$}
\label{subcap}
\end{figure}
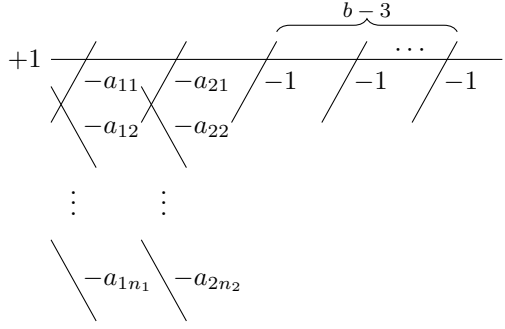

Let $\widetilde{C'}$ be a curve configuration obtained from $C'$ through standard blowing-ups. 
Then, by repeatedly applying Lemma~\ref{fundamentallem1} again, we show that a minimal symplectic filling $W$ corresponding to $C$ is obtained by a sequence of rational blowdowns from a minimal symplectic filling $\widetilde{W}$ that corresponds to the curve configuration $\widetilde{C'}$. 
 However, since $K'$ itself is a concave cap for a lens space $L$, a minimal symplectic filling 
 $\widetilde{W}$ is obtained from the minimal resolution of a singularity corresponding to $Y$ by replacing the minimal resolution of a cyclic quotient singularity corresponding to $L$ with its minimal symplectic filling $W'$ corresponding to $C'$ by Lemma~\ref{fundamentallem3}. 
 As it is known that every minimal symplectic filling of $L$ is obtained from the minimal resolution by a sequence of rational blowdowns~\cite{BOz}, we have a sequence of rational blowdowns from the minimal resolution to $W$, as desired. Especially, we can say $W$ is obtained by replacing disjoint linear subgraphs of $\Gamma$, containing a subgraph consisting of the first and the second arm together with the central vertex, with their minimal symplectic fillings (cf. Section 4 in ~\cite{CP2}).
 
 \subsection{Proof for type C}
 We prove Theorem~\ref{thm2} for a curve configuration $C$ of type C by induction on the number of $e_i$s blown-up to obtain $C$ from a configuration $(b)$ or $(c)$ in Figure~\ref{startingposition}. If we blow up at most one $e_i$ to obtain a curve configuration $C$ from $(b)$ or $(c)$ in Figure~\ref{startingposition}, then $C$ is of type B, which is proven. 
 
Now, we prove the case of type $C$  inductively. Let $C$ be a curve configuration obtained from $(b)$ or $(c)$ in Figure~\ref{startingposition} by blowing up $m$ $e_i$s. To reiterate, without loss of generality, we assume that the first $(m+1)$ arms of $(b)$ or $(c)$ become the first $(m+1)$ arms of $K$ and that the proper transforms of $e_i$ $(m+1 \leq i \leq n-1)$ are not irreducible components of $K$. Unlike for type B, we cannot obtain the first $m$ arms of $K$ without blowing up $e_{m}$. Instead, by rearranging the order of blowing-ups from $(b)$ or $(c)$ to $C$, we can obtain 
a configuration $C'$ containing $K'\leq K$ whose first $m$ arms are equal to that of $K$, except for one irreducible component $C'^1_l$ in the first arm of $K'$ with other arms of single $(-1)$ strands. The proper transforms of $e_i$ $(m\leq i \leq n -1)$ remain exceptional strands that only intersect $C'^1_l$ and single $(-1)$ arms in $C'$. Note that $a_{1l}> a'_{1l}$, where 
$-a_{1l}$ and $-a'_{1l}$ are the degrees of the $l^{\text{th}}$ component in the first arms of $K$ and $K'$, respectively. The first $m$ arms of $K'$ with the proper transform of $e_m$ can be illustrated,  as in Figure~\ref{case2c'}. The left-hand and right-hand figures are based on $(b)$ and $(c)$ in Figure~\ref{startingposition}, respectively.

\begin{figure}[h]
\begin{center}
\begin{tikzpicture}[xscale=1.1,yscale=.7]
\begin{scope}
\draw(0,0)--(3.5,0);
\draw (0.5,0.2)--(0,-0.7);
\draw (1.8,0.2)--(1.3,-0.7);
\draw (3.2,0.2)--(2.7,-2);

\draw (0,-0.3)--(0.5,-1.2);
\draw (1.3,-0.3)--(1.8,-1.2);

\node at (0.25, -1.2) {$\vdots$};

\draw (0.5,-1.6)--(0.,-2.5);

\draw (0.,-2.)--(0.5,-2.9);

\draw (0.5,-3.3)--(0,-4.2);

\draw[red,thick, dashdotted] (.2,-1.7)--(1.35,-1.74);
\draw[red,thick, dashdotted] (3,-1.8)--(1.75,-1.76);
\node[red,below] at (3,-1.8){$e_m$};

\draw (1.8,-3.3)--(1.3,-4.2);

\draw (0,-3.8)--(0.5,-4.7);
\draw (1.3,-3.8)--(1.8,-4.7);

\node at (0.25, -4.7) {$\vdots$};
\node at (1.55, -4.7) {$\vdots$};

\draw (0,-5.1)--(0.5,-6);
\draw (1.3,-5.1)--(1.8,-6);

\filldraw (1.55,-1.65) circle (0.25pt);
\filldraw (1.55,-1.75) circle (0.25pt);
\filldraw (1.55,-1.85) circle (0.25pt);
\filldraw (1.55,-1.95) circle (0.25pt);
\filldraw (1.55,-2.05) circle (0.25pt);
\filldraw (1.55,-2.15) circle (0.25pt);
\filldraw (1.55,-2.25) circle (0.25pt);
\filldraw (1.55,-2.35) circle (0.25pt);
\filldraw (1.55,-2.45) circle (0.25pt);
\filldraw (1.55,-2.55) circle (0.25pt);
\filldraw (1.55,-2.65) circle (0.25pt);
\filldraw (1.55,-2.75) circle (0.25pt);



\node at (0.25, -2.9) {$\vdots$};

\node[left] at (0,0) {$+1$};
\node[right] at (0.35, -0.25) {$-a_{11}$};
\node[right] at (1.65, -0.25) {$-a_{m1}$};
\node[right] at (3.1, -0.35) {$-1$};

\node[right] at (0.35, -0.95) {$-a_{12}$};
\node[right] at (1.65, -0.95) {$-a_{m2}$};

\node[right] at (0.35, -2.15) {$-a_{1l}'$};
\node[left] at (0.25, -1.65) {$C'^1_{l}$};


\node[right] at (0.25, -5.45) {$-a_{1n_1}$};
\node[right] at (1.55, -5.45) {$-a_{mn_m}$};
\filldraw (0.7-.05,-4.2) circle (1pt);
\filldraw (0.9-.05,-4.2) circle (1pt);
\filldraw (1.1-.05,-4.2) circle (1pt);

\end{scope}
\begin{scope}[shift={(5.2,0)}]
\draw(0,0)--(3.5,0);
\draw (0.5,0.2)--(0,-0.7);
\draw (1.8,0.2)--(1.3,-0.7);
\draw (3.2,0.2)--(2.7,-2.5);

\draw[red,thick, dashdotted](3,-2.4)--(1.75,-1.44);
\draw[red,thick, dashdotted](1.35,-1.15)--(0.1,-0.2);
\node[red,below] at (3,-2.4){$e_m$};

\draw (0,-0.3)--(0.5,-1.2);
\draw (1.3,-0.3)--(1.8,-1.2);

\node at (0.25, -1.3) {$\vdots$};
\node at (1.55, -1.3) {$\vdots$};

\draw (0.5,-2)--(0,-2.9);
\draw (1.8,-2)--(1.3,-2.9);

\draw (0.,-2.5)--(0.5,-3.4);
\draw (1.3,-2.5)--(1.8,-3.4);

\filldraw (0.25,-3.7) circle (0.25pt);


\node[left] at (0,0) {$+1$};
\node[right] at (0.35, -0.25) {$-a_{11}'$};
\node[right] at (1.65, -0.25) {$-a_{m1}$};
\node[right] at (3.1, -0.35) {$-1$};

\node[right] at (1.65, -0.95) {$-a_{m2}$};

\draw (0,-5.1)--(0.5,-6);
\draw (1.3,-5.1)--(1.8,-6);
\filldraw (0.7,-4.2) circle (1pt);
\filldraw (0.9,-4.2) circle (1pt);
\filldraw (1.1,-4.2) circle (1pt);

\node[right] at (0.25, -5.45) {$-a_{1n_1}$};
\node[right] at (1.55, -5.45) {$-a_{mn_m}$};

\filldraw (0.25,-3.7) circle (0.25pt);
\filldraw (0.25,-3.8) circle (0.25pt);
\filldraw (0.25,-3.9) circle (0.25pt);
\filldraw (0.25,-4) circle (0.25pt);
\filldraw (0.25,-4.1) circle (0.25pt);
\filldraw (0.25,-4.2) circle (0.25pt);
\filldraw (0.25,-4.3) circle (0.25pt);
\filldraw (0.25,-4.4) circle (0.25pt);
\filldraw (0.25,-4.5) circle (0.25pt);
\filldraw (0.25,-4.6) circle (0.25pt);

\filldraw (1.55,-3.7) circle (0.25pt);
\filldraw (1.55,-3.8) circle (0.25pt);
\filldraw (1.55,-3.9) circle (0.25pt);
\filldraw (1.55,-4) circle (0.25pt);
\filldraw (1.55,-4.1) circle (0.25pt);
\filldraw (1.55,-4.2) circle (0.25pt);
\filldraw (1.55,-4.3) circle (0.25pt);
\filldraw (1.55,-4.4) circle (0.25pt);
\filldraw (1.55,-4.5) circle (0.25pt);
\filldraw (1.55,-4.6) circle (0.25pt);


\end{scope}
\end{tikzpicture}
\caption{Part of intermediate configuration $C'$}
\label{case2c'}
\end{center}
\end{figure}
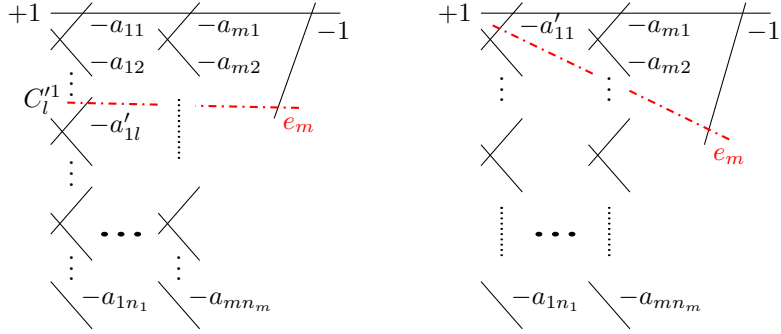

Let $\widetilde{W}$ be a minimal symplectic filling corresponding to a curve configuration $\widetilde{C'}$ obtained by $C'$ using standard blowing-ups. Then, since we do not blow up $e_i$ for $m+1\leq i \leq n-1$ to obtain $C$ from $C'$, we can show that a minimal symplectic filling $W$ corresponding to $C$ is obtained from $\widetilde{W}$ by a sequence of rational blowdowns using  Lemma~\ref{fundamentallem2} for $(m+1)^\text{th}$ arm of $K$, and Lemma~\ref{fundamentallem1} repeatedly for the other arms of $K$. 
From the construction, note that $\widetilde{C'}$ is obtained by blowing up $(m-1)$ $e_i$s. 
Therefore, there is a sequence of rational blowdowns from the minimal resolution to $\widetilde{W}$ based on the induction hypothesis, which concludes the proof.

\section{Counter examples}
In this section, we claim that the condition $N_S\leq 1$ in Theorem~\ref{thm2} is insufficient if $b=n+1$. That is, there is a family of minimal symplectic fillings with $N_S\leq 1$ that cannot be obtained via rational blowdown surgeries. Recall that the isotopy type of a symplectic line arrangement $S$ satisfying $N_S \leq 1$ is unique. Hence, a symplectic line arrangement $S$ satisfying $N_S \leq1$ is completely determined by the number of all symplectic lines in $S$ and symplectic lines passing through a unique multi-intersection point. Let $S_{n,m}$ be a symplectic line arrangement consisting of $n$ symplectic lines (except the central complex line $\mathbb{CP}^1$) that contains $m$ symplectic lines passing through a unique multi-intersection point.

Next, we consider a Seifert $3$-manifold $Y_n$ determined by a left-hand plumbing graph in Figure~\ref{Yn} whose concave cap $K_n$ is given by a right-hand figure. Here the degrees of unlabeled vertices and strands are all $-2$.

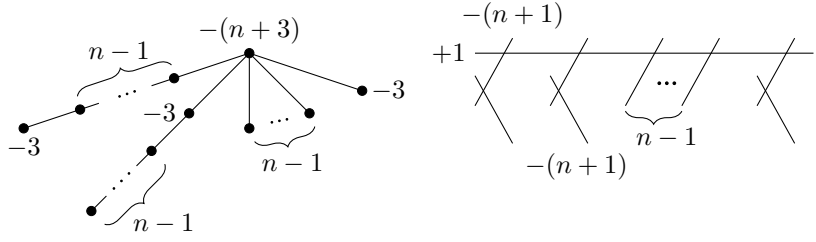
\begin{figure}[h]
\begin{tikzpicture}
\begin{scope}
\node[bullet] at (0,0){};
\node[above] at (0,0){$-(n+3)$};

\draw (0,0)--(0,-1) (0,0)--(.8,-.8) (0,0)--(1.5,-.5);
\node[right] at (1.5,-.5){$-3$};

\draw(0,0)--(-1.3,-1.3/3) (-2,-2/3)--(-3,-.5-1/3-.5/3) ;
\node[bullet] at (-1,-1/3){};
\node[bullet] at (-2.25,-.5-1/3+.25/3){};

	\draw [decorate,decoration={mirror,brace,amplitude=5pt},xshift=-1pt,yshift=5pt]
	(-1,-1/3) -- (-2.25,-.5-1/3+.25/3) node [black,midway,xshift=-2pt,yshift=11pt] 
	{$n-1$};

\node[bullet] at (-3,-.5-1/3-.5/3){};
\filldraw(-3.25/2,-3.25/6) circle (.3pt);
\filldraw(-3.25/2+.1,-3.25/6+.1/3) circle (.3pt);
\filldraw(-3.25/2-.1,-3.25/6-.1/3) circle (.3pt);

\node[below] at (-3,-.5-1/3-.5/3){$-3$};

\draw(0,0)--(-1.5,-1.5) (-1.9,-1.9)--(-2.1,-2.1);
\node[bullet] at (-.8,-.8){};
\node[left] at (-.85,-.8){$-3$};
\node[bullet] at (-1.3,-1.3){};
\node[bullet] at (-2.1,-2.1){};

	\draw [decorate,decoration={brace,amplitude=5pt},xshift=5pt,yshift=-5pt]
	(-1.3,-1.3) -- (-2.1,-2.1) node [black,midway,xshift=11pt,yshift=-11pt] 
	{$n-1$};

\filldraw(-1.6,-1.6) circle (.3pt);
\filldraw(-1.7,-1.7) circle (.3pt);
\filldraw(-1.8,-1.8) circle (.3pt);
\node[bullet] at (.8,-.8){};
\node[bullet] at (0,-1.){};
\filldraw(.3,-.92) circle (.3pt);
\filldraw(.4,-.9) circle (.3pt);
\filldraw(.5,-.88) circle (.3pt);

	\draw [decorate,decoration={mirror,brace,amplitude=5pt},xshift=2pt,yshift=-5pt]
	(0,-1) -- (.8,-.8) node [black,midway,xshift=3pt,yshift=-11pt] 
	{$n-1$};

\node[bullet] at (1.5,-.5){};
\end{scope}
\begin{scope}[shift={(3,0)}]
\draw(0,0)--(4.5,0);

\node[above] at (0.5,0.2) {$-(n+1)$};
\draw (0.5,0.2)--(0,-0.7);
\draw (1.25+.25,0.2)--(0.75+.25,-0.7);

\draw (2.25+.25,0.2)--(1.75+.25,-0.7);

\filldraw (2.55,-0.4) circle (.5pt);
\filldraw (2.45,-0.4) circle (.5pt);
\filldraw (2.65,-0.4) circle (.5pt);

\draw (3.25,0.2)--(2.75,-0.7);

\draw (4.25,0.2)--(3.75,-0.7);

\draw (0,-0.3)--(0.5,-1.2);
\draw (0.75+.25,-0.3)--(1.25+.25,-1.2);
\node[below] at (1.25+.1,-1.2) {$-(n+1)$};
\draw (3.75,-0.3)--(4.25,-1.2);

	\draw [decorate,decoration={brace,amplitude=5pt},xshift=0pt,yshift=-2pt]
	(2.75,-0.7) -- (1.75+.25,-0.7) node [black,midway,xshift=5pt,yshift=-10pt] 
	{$n-1$};

\node[left] at (0,0) {$+1$};

\end{scope}
\end{tikzpicture}
\caption{Plumbing graph of $Y_n$ and its concave cap $K_n$}
\label{Yn}
\end{figure}

Let $W_n$ be a minimal symplectic filling of $(Y_n,\xi_{\text{can}})$ corresponding to a curve configuration $C_n$ obtained from a symplectic line arrangement $S_{n+2,n+1}$, as it follows.
We first obtain the configuration $S'_{n+2,n+1}$ in Figure~\ref{blowingup1} by blowing up all intersection points between symplectic lines in  $S_{n+2,n+1}$. 
Since there is no arm in $K$ starting with a $(-1)$ strand, unlike with the proof of Theorem~\ref{thm2}, we can blow up at an exceptional curve $e$ to obtain a curve configuration for $Y_n$. We blow up all intersection points of $e$ except one with the second arm of $S'_{n+2,n+1}$ to obtain $C^2_2$ in $K$. Then, we blow up an intersection point between $e_1$ and the second arm of $S'_{n+2,n+1}$ for $C^1_2$ of $K$ and blow up an intersection point between $e_{n+1}$ and the first arm of $S'_{n+2,n+1}$ for $C^{n+2}_2$ of $K$, resulting in the curve configuration $C_n$ in Figure~\ref{blowingup1}. Note that we do not illustrate the proper transforms of $e_2,\dots, e_n$ in the curve configuration $C_n$ for convenience.

\begin{figure}[h]
\begin{tikzpicture}[scale=0.95]
\begin{scope}[scale=1.3]

\filldraw (1.3,0.4) circle (.4pt);
\filldraw (1.4,0.4) circle (.4pt);
\filldraw (1.5,0.4) circle (.4pt);

\draw (0,0.2)--(2,0.2);
\draw (0.4,0.5)--(0.4,-1.6);

\draw (0.8,0.5)--(0.8,-0.6);
\draw (1.2,0.5)--(1.2,-0.8);
\draw (1.6,0.5)--(1.6,-1);

\draw[red, thick, dashdotted] (0.6,0)--(1.8,0);
\node[left, red] at (2.2,0) {$e$};
\draw[red, thick, dashdotted](0.3,-0.6)--(0.9,-0.3);
\draw[red, thick, dashdotted](0.3,-1)--(1.3,-0.5);
\draw[red, thick, dashdotted](0.3,-1.4)--(1.7,-0.7);
\node[left, red] at (0.3,-0.6) {$e_1$};
\node[left, red] at (0.3,-1) {$e_{2}$};
\node[left, red] at (0.3,-1.4) {$e_{n+1}$};


\node at (1,-2) {$S'_{n+2,n+1}$};

\filldraw (0.8,-.95+.05) circle (.4pt);
\filldraw (0.85,-1.+.05) circle (.4pt);
\filldraw (0.9,-1.05+.05) circle (.4pt);
\node[above] at (0.25,0.5) {$-n$};
\node[above] at (0.8,0.5) {$-1$};
\node[above] at (1.2,0.5) {$-1$};
\node[above] at (1.6,0.5) {$-1$};
\node[left] at (0,0.2) {$+1$};

\filldraw (1.2,0.) circle (1pt);
\filldraw (1.6,0.) circle (1pt);

\filldraw (.4,-1.35) circle (1pt);
\filldraw (.8,-0.35) circle (1pt);

\draw[->,very thick] (2.25,-.5)--(2.75,-.5);

\end{scope}
\begin{scope}
\end{scope}

\begin{scope}[shift={(4.5,0)}]
\draw(0,0)--(4.5,0);

\node[above] at (0.5,0.2) {$-(n+1)$};
\draw (0.5,0.2)--(0,-0.7)--(-.5-.1,-1.6-.1*9/5);

\draw[red, thick, dashdotted](-.7,-1.5-.15)--(4.35,-1.05-.225);
\draw (1.25+.25,0.2)--(0.75+.25,-0.7);
\draw[red, thick, dashdotted](0.05,-0.7-.15)--(1.5,-.1);

\draw (2.25+.25,0.2)--(1.75+.25+.05,-0.7+.05*9/5);

\filldraw (2.5,-0.25) circle (.5pt);
\filldraw (2.6,-0.25) circle (.5pt);
\filldraw (2.7,-0.25) circle (.5pt);

\draw (3.25,0.2)--(2.75,-0.7);

\draw (4.25,0.2)--(3.75-.1,-0.7-.1*9/5);

\draw (0,-0.3-.1-.1)--(0.5,-1.2-.1-.1);
\draw (0.75+.25,-0.3-.1-.1)--(1.25+.25-.1,-1.2-.1-.1+.1*9/5);

\draw[red, thick, dashdotted](1.04,-.75)--(2.3,-.4);
\draw[red, thick, dashdotted](1.09,-1.)--(3,-.55);
\draw[red, thick, dashdotted](1.17,-1.2)--(3.8,-.75);

\draw (3.75,-0.3-.1-.1)--(4.25,-1.2-.1-.1);

	\draw [decorate,decoration={brace,amplitude=5pt},xshift=0pt,yshift=2pt]
	(2.5,0.2) -- (3.25,0.2) node [black,midway,xshift=0pt,yshift=10pt] 
	{$n-1$};


\node[left] at (0,0) {$+1$};
\node at (2,-2*13/10) {$C_n$};

\draw[red, thick, dashdotted](0.05,-0.7-.15)--(1.5,-.1);
\draw[red, thick, dashdotted](-.7,-1.5-.15)--(4.35,-1.05-.225);

\end{scope}
\end{tikzpicture}
\caption{Curve configuration $C_n$ obtained from $S_{n+2,n+1}$}
\label{blowingup1}
\end{figure}
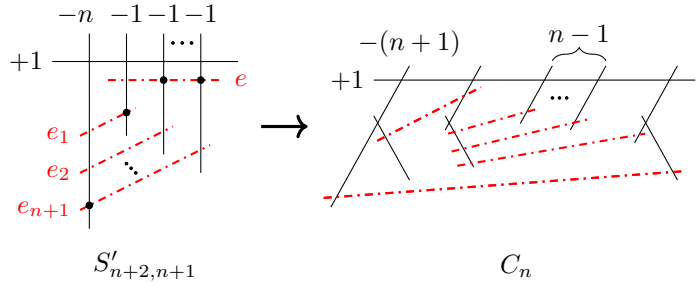

\begin{thm}
\label{thm3}
For each $n\geq 3$, the minimal symplectic filling $W_n$ of $(Y_n,\xi_{\text{can}})$ cannot be obtained by a sequence of rational blowdowns from the minimal resolution of the corresponding weighted homogeneous surface singularity.
\end{thm}


To prove Theorem~\ref{thm3}, we first observe the effect on symplectic line arrangements under a single rational blowdown surgery.

\begin{lem}
\label{counterexamplelem1}
Assume that a minimal symplectic filling $W$ is obtained from $W'$ by rationally blowing down 
$G\subset W'$. If a symplectic line arrangement corresponding to $W'$ is $S_{n,m}$, then a symplectic line arrangement corresponding to $W$ is either $S_{n,m}$ or $S_{n,m-1}$. 
\end{lem}

\begin{proof}
In the proof of Proposition~\ref{prop2}, we showed that $G$ is obtained by blowing-ups from a single exceptional curve $e$. If $e$ corresponds to a non-multi-intersection point, the corresponding symplectic line arrangement does not change during surgery. If $e$ corresponds to a unique multi-intersection point of $S_{n,m}$, then the symplectic line arrangement corresponding to $W$ is $S_{n,m}$ or $S_{n,m-1}$ depending on whether a symplectic line arrangement corresponding to the rational homology ball filling of $(\partial G, \xi_{\text{can}})$ with respect to $K_G$ is $S_{m,m}$ or $S_{m,m-1}$.
\end{proof}

\begin{proof}[Proof of Theorem~\ref{thm3}]
We assume that there is a sequence of rational blowdowns from the minimal resolution to $W_n$. Then there exists a minimal symplectic filling $W'_n$ of $Y_n$ such that $W_n$ is obtained from $W'_n$ by rationally blowing down $G_n\subset W'_n$. Furthermore, $W'_n$ itself is also obtained by a sequence of rational blowdowns so that  the corresponding symplectic line arrangement to $W'_n$ is $S_{n+2,n+2}$ or $S_{n+2,n+1}$, by Lemma~\ref{counterexamplelem1}.

First, we consider the curve configurations obtained from $S_{n+2,n+2}$. Note that each curve configuration obtained from $S_{n+2,n+2}$ is of type A or type B, as described in Section~\ref{4}. Because of the degrees appeared in $K_n$, we can only have curve configurations of type A for the minimal symplectic fillings of $Y_n$. Furthermore, there is only one curve configuration $\widetilde{S}_{n+2,n+2}$ of type A, standard blowing-ups of $S_{n+2,n+2}$, which corresponds to the minimal resolution of the corresponding singularity. 
In the curve configuration $\widetilde{S}_{n+2,n+2}$, each homology class of the $(-1)$ pseudo-holomorphic curves appears in only one arm of $K_n$ except for a pseudo-holomorphic curve $e$ corresponding to a unique multi-intersection point of $S_{n+2,n+2}$. Therefore, if a minimal symplectic filling $W_n$ is obtained from the minimal resolution by a single rational blowdown $G_n$, there is at least one $(-1)$ pseudo-holomorphic curve in the curve configuration $C_n$ of $W_n$ whose homology class appears in only one arm of $K_n$ unless $W_n$ is a rational homology ball filling. 
However, the homology class of every $(-1)$ pseudo-holomorphic curve in $C_n$ appears in at least two arms of $K_n$, and $W_n$ is not a rational homology ball filling unless $n=1$.

Next, we show that a curve configuration $C'_n$ of $W'_n$ cannot be obtained from $S_{n+2,n+1}$. Since we should blow all intersection points among the symplectic lines of a symplectic line arrangement to obtain a curve configuration, all curve configurations obtained from $S_{n+2,n+1}$ for minimal symplectic fillings of $Y_n$ are actually obtained from $S'_{n+2,n+1}$ by blowing-ups (Figure~\ref{blowingup1}). We can divide all curve configurations obtained from $S_{n+2,n+1}$ into two types: those with and without blowing up at an exceptional curve $e$.

We first assume that $C'_n$ is obtained from $S'_{n+2,n+1}$ without blowing up at $e$. 
Thus, the homological data of $K_n$ regarding $e$ in $C'_n$ is different from that of $K_n$ regarding $e$ in $C_n$.
Since only the homology classes of $E_{G_n}$ can change the homological data of $K_n$ for $W'_n$ under rationally blowing down $G_n\subset W'_n$, 
a symplectic embedding of $G_n$ in $W'_n$  should be obtained from $e$ (refer to the proof of Proposition~\ref{prop2}; we blow up all intersection points of $e$ to obtain a symplectic embedding of $G_n$ from $e$), and the homology classes of $e_i$'s in $S'_{n+2,n+1}$ do not belong to $E_{G_n}$. 
Here, $E_{G_n}$ denotes the set of $(-1)$ pseudo-holomorphic curves whose homology classes appear in the irreducible components of $G_n$.
Furthermore, since we blow up two $e_i$s to obtain $C_n$ from $S'_{n+2,n+1}$, our observation implies that we should also blow up the two $e_i$s to obtain $C'_n$ resulting from the configuration $S''_{n+2,n+1}$  in Figure~\ref{middle1}.
Then, the second arm of $S''_{n+2,n+1}$ becomes an arm in $K_n$ consisting of a single $(-2)$ strand in $C'_n$ because we do not blow up at intersection points of $e$ to obtain $C'_n$. This implies that there is no way of obtaining an embedding $G_n$ in $W'_n$ from  $e$ by blowing-ups because $e$ intersects the single $(-2)$ arm of $K$ in $W'_n$ so that we cannot blow it up to make $e$ disjoint from $K_n$.

\begin{figure}[h]
\begin{tikzpicture}
\begin{scope}[scale=1.35]

\filldraw (1.3,0.4) circle (.4pt);
\filldraw (1.4,0.4) circle (.4pt);
\filldraw (1.5,0.4) circle (.4pt);

\draw (0,0.2)--(2.4,0.2);
\draw (0.4,0.5)--(0.4,-1.6);

\draw (0.8,0.5)--(0.8,-0.6);
\draw (1.2,0.5)--(1.2,-0.8);
\draw (1.6,0.5)--(1.6,-1);
\draw (2.,0.5)--(2,-1.2);

\draw (2.1,-.9)--(1.6,-1.6);
\node[right] at (2.1,-.9) {$-2$};
\draw[red, thick, dashdotted] (0.3,-1.5)--(1.8,-1.5);

\draw[red, thick, dashdotted] (0.6,0)--(2.2,0);

\draw[red, thick, dashdotted](0.5,-0.5)--(0.9,-0.3);
\draw (0.7,-0.5)--(0.3,-0.3);
\node[left] at (0.3,-0.3) {$-2$};

\draw[red, thick, dashdotted](0.3,-1)--(1.3,-0.5);
\draw[red, thick, dashdotted](0.3,-1.4)--(1.7,-0.7);



\filldraw (0.8,-.95+.05) circle (.4pt);
\filldraw (0.85,-1.+.05) circle (.4pt);
\filldraw (0.9,-1.05+.05) circle (.4pt);
\node[below] at (0.4,-1.6) {$-(n+1)$};
\node[above] at (0.8,0.5) {$-2$};
\node[above] at (1.2,0.5) {$-1$};
\node[above] at (1.6,0.5) {$-1$};
\node[above] at (2,0.5) {$-1$};
\node[left] at (0,0.2) {$+1$};


\end{scope}
\end{tikzpicture}
\caption{Configuration $S''_{n+2,n+1}$}
\label{middle1}
\end{figure}

Next, we assume that $C'_n$ is obtained from $S'_{n+2,n+1}$ by blowing up the intersection points on $e$. Then the proper transform of $e$ is an irreducible component of $K_n$ in $C'_n$. Hence, we should blow up at least all intersection points on $e$ except one, as we obtain a curve configuration $C_n$ from $S'_{n+2,n+1}$. Since the length of each arm in $K_n$ is at most two, we should also blow up the intersection points on $e$ exactly as before, so that the first two arms in the resulting configuration $S'''_{n+2,n+1}$ (refer to Figure~\ref{middle}) become the first and second arms of $K_n$. 
Note that we need the condition $n\geq 3$ to guarantee that the first two arms of $S'''_{n+2,n+1}$ become the first two arms of $K_n$ in $C'_n$.
Then, because of the degrees in $K_n$, we should reblow up again an exceptional strand in $S'''_{n+2,n+1}$ coming from one of $e_{i}$'s in $S'_{n+2,n+1}$ for $C^{n+2}_2$ of $K_n$ and an  exceptional strand from $e_1$ for $C^{1}_2$ of $K_n$, so that the resulting curve configuration is equivalent to $C_n$, which contradicts the assumption. 

\begin{figure}[h]
\begin{tikzpicture}[scale=0.9]
\begin{scope}[shift={(4.5,0)}]
\draw(0,0)--(4.5,0);

\node[above] at (0.5,0.2) {$-n$};
\node[above] at (1.5,0.2) {$-1$};
\node[above] at (2.5,0.2) {$-2$};
\node[above] at (3.25,0.2) {$-2$};
\node[above] at (4.25,0.2) {$-2$};

\draw (0.5,0.2)--(0,-0.7)--(-.5-.1-.4*5/9,-1.6-.1*9/5-.4);

\draw (1.25+.25,0.2)--(0.75+.25,-0.7);
\draw[red, thick, dashdotted](0.05,-0.2)--(1.5,-.2);

\draw (2.25+.25,0.2)--(1.75+.25+.05,-0.7+.05*9/5);

\filldraw (2.5,-0.25) circle (.5pt);
\filldraw (2.6,-0.25) circle (.5pt);
\filldraw (2.7,-0.25) circle (.5pt);

\draw (3.25,0.2)--(2.75,-0.7);

\draw (4.25,0.2)--(3.75-.1-.9*5/9-.4*5/9,-0.7-.1*9/5-.9-.4);

\draw[red, thick, dashdotted](-.9,-0.7-.1*9/5-.9-.2)--(3.4,-0.7-.1*9/5-.9-.2);

\draw (0.75+.25,-0.3-.1-.1)--(1.25+.25-.1,-1.2-.1-.1+.1*9/5);
\node[below] at (1.25+.25-.1,-1.2-.1-.1+.1*9/5) {$-(n+1)$};

\draw[red, thick, dashdotted](1.04,-.75)--(2.3,-.4);
\draw[red, thick, dashdotted](1.09,-1.)--(3,-.55);
\draw[red, thick, dashdotted](1.17,-1.2)--(3.8,-.75);



\node[left] at (0,0) {$+1$};


\end{scope}
\end{tikzpicture}

\caption{Configuration $S'''_{n+2,n+1}$}
\label{middle}
\end{figure}
In conclusion, there is no minimal symplectic filling $W'_n$ of $Y_n$ such that $W_n$ is obtained from $W'_n$ by a single rational blowdown surgery. Hence, $W_n$ cannot be obtained by a sequence of rational blowdowns from the minimal resolution.

\end{proof}

\newpage
\appendix

\section{Minimal symplectic fillings versus Milnor fibers of weighted homogeneous surface singularities}
\centerline{Hakho Choi, Jongil Park and Jaekwan Jeon}

\bigskip

In this appendix, we compare minimal symplectic fillings of a Seifert $3$-manifold $Y(-b; (\alpha_1, \beta_1), (\alpha_2, \beta_2),\ldots, (\alpha_n, \beta_n))$ with $b \geq n+2$ and Milnor fibers of a weighted homogeneous surface singularity $(X,0)$ corresponding to $Y$. 
As we mentioned in the Introduction, every Milnor fiber of $(X,0)$ gives a minimal symplectic filling of $Y$. Therefore, a question is whether all minimal symplectic fillings come from Milnor fibers of $(X,0)$ or not. To deal with the question, we consider special partial resolutions of $(X,0)$, so-called $P$-resolutions. The notion of $P$-resolution is originally given by Koll{\'a}r–Shepherd-Barron~\cite{KSB} to analyze the versal deformation space of a quotient surface singularity, which can also be defined for weighted homogeneous surface singularities. Topologically, a Milnor fiber corresponding to a given $P$-resolution is obtained by a sequence of blowing-ups and rational blowdowns from the minimal resolution of $(X,0)$. In many cases as well as quotient surface singularities, the sequence of blowing-ups and rational blowdowns can be interpreted as a sequence of rational blowdowns along chains of symplectic spheres~\cite{CP1}. Hence, when a minimal symplectic filling $W$ is obtained from a sequence of rational blowdowns from the minimal resolution of $(X,0)$ corresponding to $Y$, it is natural to find a $P$-resolution whose Milnor fiber is diffeomorphic to $W$.  
As the first step for this, we construct a partial resolution of $(X,0)$ such that a $\mathbb{Q}$-Gorenstein smoothing of singularities of class $T$ gives a minimal symplectic filling diffeomorphic to $W$, which was already obtained by a sequence of rational blowdowns from the minimal resolution given in Section~\ref{4}. And then, we check the ample condition to show that the partial resolution we constructed is actually a $P$-resolution. Finally,
combining our main criterion (Theorem~\ref{thm2}) for minimal symplectic fillings to be obtained from a sequence of rational blowdowns, we get the following result.

\begin{thm}
For a Seifert $3$-manifold $Y(-b; (\alpha_1, \beta_1), (\alpha_2, \beta_2),\ldots, (\alpha_n, \beta_n))$ with $b\geq n+2$, any minimal symplectic filling $W$ of $Y$ with $N_S\leq 1$ is realized as a Milnor fiber of some $P$-resolution of $(X,0)$, a weighted homogeneous surface singularity corresponding to $Y$.
\label{a}
\end{thm}

Note that,  if $b\geq n+3$, every minimal symplectic filling satisfies automatically $N_S\leq 1$. Hence, as a corollary, we easily get

\begin{cor}
For a Seifert $3$-manifold $Y(-b; (\alpha_1, \beta_1), (\alpha_2, \beta_2),\ldots, (\alpha_n, \beta_n))$ with $b\geq n+3$,
every minimal symplectic filling $W$ of $Y$ is realized as a Milnor fiber of some $P$-resolution of $(X,0)$.
\end{cor}

Before we prove Theorem~\ref{a} above, we briefly review the notion of $P$-resolution.
\begin{defn}
A normal surface singularity is \emph{of class $T$} if it is a rational double point singularity or a cyclic quotient surface singularity of type $\frac{1}{dn^2}(1, dna-1)$ with $d \ge 1$, $n \ge 2$, $1 \le a < n$, and $(n,a)=1$. 
\end{defn}

Note that one-parameter $\mathbb{Q}$-Gorenstein smoothing of a singularity of class $T$ is interpreted topologically as a rational blowdown surgery. Furthermore, thanks to J. Wahl~\cite{W}, a cyclic quotient surface singularity of class $T$ can be recognized from its minimal resolution as follows:
\vspace{-5pt}
\begin{prop}
\begin{enumerate}
\item The singularities  
\begin{tikzpicture} 
\filldraw (0,0) circle (2pt); 
\draw (0,0) node[above] {$-4$}; 
\end{tikzpicture} and 
\begin{tikzpicture} 
\filldraw (0,0) circle (2pt); 
\filldraw (1,0) circle (2pt); 

\filldraw (3,0) circle (2pt); 
\filldraw (4,0) circle (2pt); 

\draw (0,0)--(1.5,0) (2.5,0)--(4,0);
\draw (2,0) node {$\dots$};
\draw (0,0) node[above] {$-3$}; 
\draw (1,0) node[above] {$-2$}; 
\draw (3,0) node[above] {$-2$}; 
\draw (4,0) node[above] {$-3$}; 

\end{tikzpicture}
are of class $T$.

\item If 
\begin{tikzpicture} 
\filldraw (0,0) circle (2pt); 
\filldraw (1,0) circle (2pt); 

\filldraw (3,0) circle (2pt); 
\filldraw (4,0) circle (2pt); 

\draw (0,0)--(1.5,0) (2.5,0)--(4,0);
\draw (2,0) node {$\dots$};
\draw (0,0) node[above] {$-b_1$}; 
\draw (1,0) node[above] {$-b_2$}; 
\draw (3,0) node[above] {$-b_{r-1}$}; 
\draw (4,0) node[above] {$-b_r$};

\end{tikzpicture}
is of class $T$, so are $$\begin{tikzpicture}[scale=1.3] 
\filldraw (0,0) circle (2pt); 
\filldraw (1,0) circle (2pt); 

\filldraw (3,0) circle (2pt); 
\filldraw (4,0) circle (2pt); 

\draw (0,0)--(1.5,0) (2.5,0)--(4,0);
\draw (2,0) node {$\dots$};
\draw (0,0) node[above] {$-2$}; 
\draw (1,0) node[above] {$-b_1$}; 
\draw (3,0) node[above] {$-b_{r-1}$}; 
\draw (4,0) node[above] {$-(b_r+1)$};

\end{tikzpicture}$$
and
$$\begin{tikzpicture} [scale=1.3]
\filldraw (0,0) circle (2pt); 
\filldraw (1,0) circle (2pt); 

\filldraw (3,0) circle (2pt); 
\filldraw (4,0) circle (2pt); 

\draw (0,0)--(1.5,0) (2.5,0)--(4,0);
\draw (2,0) node {$\dots$};
\draw (0,0) node[above] {$-(b_1+1)$}; 
\draw (1,0) node[above] {$-b_2$}; 
\draw (3,0) node[above] {$-b_{r}$}; 
\draw (4,0) node[above] {$-2$};

\end{tikzpicture}$$
\item Every singularity of class $T$ that is not a rational double point can be obtained directly from one of the singularities described in (1) and by iterating through the steps described in (2) above.
\end{enumerate}
\label{inductiveofsingT}
\end{prop}



\begin{defn}
A \emph{$P$-resolution} $f: (Z,E)\rightarrow (X,0)$ of a weighted homogeneous surface singularity $(X,0)$ is a partial resolution such that $Z$ has at most rational double points or singularities of class $T$ and $K_Z$ is ample relative to $f$.
\end{defn}

 We usually describe a $P$-resolution $Z\rightarrow X$ as the minimal resolution $\pi:\widetilde{Z}\rightarrow Z$ of $Z$ with $\pi$-exceptional divisors. Note that the ample condition in the definition of a $P$-resolution is equivalent to the discrepancy condition on each $(-1)$ curve on $\widetilde{Z}$: Every $(-1)$ curve on $\widetilde{Z}$ must intersect two curves $E_1$ and $E_2$, which are exceptional for singularities of class $T$ on $Z$, so that the sum of the $k_i$ coefficients of $E_i$ in the canonical divisor $K_{\widetilde{Z}}$ must be less than $-1$. 
 
\vspace{10pt}

 Now we are ready to prove Theorem~\ref{a}. As the first step, we construct a partial resolution of $(X, 0)$ corresponding to a minimal symplectic filling $W$ obtained by a sequence of rational blowdowns from the minimal resolution of $(X, 0)$.

\begin{prop}
Let $Y$ be a Seifert $3$-manifold $Y(-b; (\alpha_1, \beta_1), \ldots, (\alpha_n, \beta_n))$ with $b\geq n+2$ and $(X,0)$ be a weighted homogeneous singularity corresponding to $Y$.
Then, for a minimal symplectic filling $W$ with $N_S\leq 1$, there is a partial resolution $f: (Z,E)\rightarrow (X,0)$ with only rational double points or singularities of class $T$ such that a Milnor fiber of the $\mathbb{Q}$-Gorenstein smoothing of $(Z, E)$ is diffeomorphic to $W$.
\label{partial}
\end{prop}
\vspace{-17pt}
\begin{proof}
Recall that we divide curve configurations corresponding to minimal symplectic fillings of $Y$ with $N_S\leq 1$ into the following three types in Section ~\ref{4}.
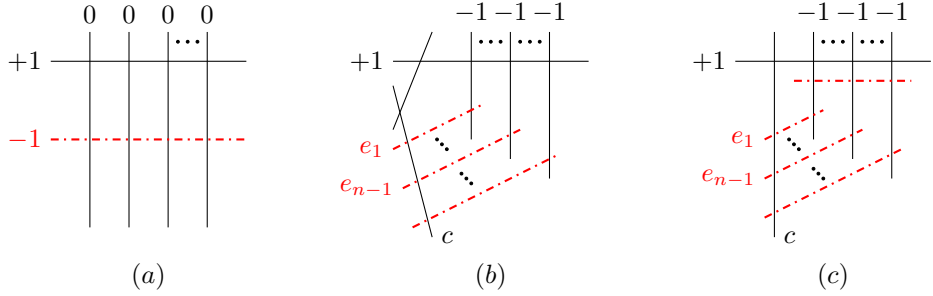
\begin{figure}[h]
\begin{center}
\begin{tikzpicture}[scale=1.3]
\begin{scope}[shift={(-3.5,0)}]

\draw (0,0.2)--(2,0.2);
\draw (0.4,0.5)--(0.4,-1.5);
\draw (0.8,0.5)--(0.8,-1.5);
\draw (1.2,0.5)--(1.2,-1.5);
\draw (1.6,0.5)--(1.6,-1.5);
\draw[red, thick, dashdotted] (0,-0.6)--(2,-0.6);
\node[left,red] at (0,-0.6) {$-1$};
\node[left] at (0,0.2) {$+1$};
\node at (1,-2) {$(a)$};

\filldraw (1.3,0.4) circle (.4pt);
\filldraw (1.4,0.4) circle (.4pt);
\filldraw (1.5,0.4) circle (.4pt);
\node[above] at (0.4,0.5) {$0$};
\node[above] at (0.8,0.5) {$0$};
\node[above] at (1.2,0.5) {$0$};
\node[above] at (1.6,0.5) {$0$};
\end{scope}

\begin{scope}
\node at (1,-2) {$(b)$};

\draw (0,0.2)--(2,0.2);
\draw (0.4,0.5)--(0,-0.5);
\draw (0,-0.05)--(0.4,-1.6);
\node[right] at (0.4,-1.6) {$c$};
\draw (0.8,0.5)--(0.8,-0.6);
\draw (1.2,0.5)--(1.2,-0.8);
\draw (1.6,0.5)--(1.6,-1);

\draw[red, thick, dashdotted](0,-0.7)--(0.9,-0.25);
\draw[red, thick, dashdotted](0.1,-1.1)--(1.3,-0.5);
\draw[red, thick, dashdotted](0.2,-1.5)--(1.7,-0.75);

\node[left, red] at (0,-0.7) {$e_1$};
\node[left, red] at (0.1,-1.1) {$e_{n-1}$};

\filldraw (1.3-.4,0.4) circle (.4pt);
\filldraw (1.4-.4,0.4) circle (.4pt);
\filldraw (1.5-.4,0.4) circle (.4pt);

\filldraw (1.3,0.4) circle (.4pt);
\filldraw (1.4,0.4) circle (.4pt);
\filldraw (1.5,0.4) circle (.4pt);
\node[above] at (0.8,0.5) {$-1$};
\node[above] at (1.2,0.5) {$-1$};
\node[above] at (1.6,0.5) {$-1$};
\node[left] at (0,0.2) {$+1$};

\filldraw (0.7-.25,-.95+.35) circle (.4pt);
\filldraw (0.75-.25,-1.+.35) circle (.4pt);
\filldraw (0.8-.25,-1.05+.35) circle (.4pt);

\filldraw (0.7,-.95) circle (.4pt);
\filldraw (0.75,-1.) circle (.4pt);
\filldraw (0.8,-1.05) circle (.4pt);

\end{scope}

\begin{scope}[shift={(3.5,0)}]

\filldraw (1.3-.4,0.4) circle (.4pt);
\filldraw (1.4-.4,0.4) circle (.4pt);
\filldraw (1.5-.4,0.4) circle (.4pt);

\filldraw (1.3,0.4) circle (.4pt);
\filldraw (1.4,0.4) circle (.4pt);
\filldraw (1.5,0.4) circle (.4pt);
\node at (1,-2) {$(c)$};

\draw (0,0.2)--(2,0.2);
\draw (0.4,0.5)--(0.4,-1.6);
\node[right] at (0.4,-1.6) {$c$};

\draw (0.8,0.5)--(0.8,-0.6);
\draw (1.2,0.5)--(1.2,-0.8);
\draw (1.6,0.5)--(1.6,-1);
\draw[red, thick, dashdotted] (0.6,0)--(1.8,0);
\draw[red, thick, dashdotted](0.3,-0.6)--(0.9,-0.3);
\draw[red, thick, dashdotted](0.3,-1)--(1.3,-0.5);
\draw[red, thick, dashdotted](0.3,-1.4)--(1.7,-0.7);
\node[left, red] at (0.3,-0.6) {$e_1$};
\node[left, red] at (0.3,-1) {$e_{n-1}$};

\filldraw (0.7-.15,-.95+.35) circle (.4pt);
\filldraw (0.75-.15,-1.+.35) circle (.4pt);
\filldraw (0.8-.15,-1.05+.35) circle (.4pt);

\filldraw (0.8,-.95+.05) circle (.4pt);
\filldraw (0.85,-1.+.05) circle (.4pt);
\filldraw (0.9,-1.05+.05) circle (.4pt);

\node[above] at (0.8,0.5) {$-1$};
\node[above] at (1.2,0.5) {$-1$};
\node[above] at (1.6,0.5) {$-1$};
\node[left] at (0,0.2) {$+1$};

\end{scope}
\end{tikzpicture}
\end{center}
\caption{Three configurations}
\label{startingpositionappendix}
\end{figure} 

\begin{itemize}
\item{Type A:}
Curve configurations obtained from $(a)$ in Figure~\ref{startingpositionappendix} without blowing up the exceptional strand.
\item{Type B:}
Curve configurations obtained from $(b)$ or $(c)$ in Figure~\ref{startingpositionappendix} by blowing up at most one $e_i$ $(1\leq i \leq n-1)$.
\item{Type C:}
Curve configurations obtained from $(b)$ or $(c)$ in Figure~\ref{startingpositionappendix} by blowing up at least two  $e_i$s $(1\leq i \leq n-1)$.
\end{itemize}
From the proof for type A and type B in Section~\ref{4}, we know that each minimal symplectic filling $W$ of type A or type B is obtained from the minimal resolution by replacing disjoint linear subgraphs of $\Gamma$ with their minimal symplectic fillings.  Hence we can construct a partial resolution corresponding to $W$ by using an explicit one-to-one correspondence between minimal symplectic fillings and $P$-resolutions of a cyclic quotient surface singularity~\cite{PPSU}. Hence it suffices to construct a partial resolution corresponding to a minimal symplectic filling $W$ of type C.

In order to construct such a partial resolution, we start with a 3-legged case which can be generalized to the multi-legged case. Recall that we find another minimal symplectic filling $\widetilde{W}$ of $Y$ such that $W$ is obtained from $\widetilde{W}$ by a sequence of rational blowdowns and $\widetilde{W}$ itself is obtained from the minimal resolution $\Gamma$ by a sequence of rational blowdowns. More precisely, $\widetilde{W}$ is obtained by replacing a linear subgraph consisting of two arms in $\Gamma$ together with the central vertex with its minimal symplectic filling while $W$ is obtained by replacing a symplectic embedding of a linear chain $L$ in $\widetilde{W}$ with its minimal symplectic filling (For more details, refer to Section 4 in \cite{CP2}). Type C is different from other types in a sense that $L$ is not anymore a linear subgraph in $\Gamma$.  In order to find such $L$ explicitly from the resolution graph $\Gamma$, we blow up intersection points of the central vertex as follows: Let $C$ be a curve configuration corresponding to $W$. As we saw in the proof of Theorem~\ref{thm2} for type C, we have an intermediate configuration $C'$ obtained from (b) or (c) of Figure~\ref{startingpositionappendix} by blowing-ups as in Figure~\ref{caseappendix}  before we get a curve configuration $C$.
\begin{figure}[h]
\begin{center}
\begin{tikzpicture}[xscale=1.1,yscale=.75]
\begin{scope}
\draw(0,0)--(3.5,0);
\draw (0.5,0.2)--(0,-0.7);
\draw (1.8,0.2)--(1.3,-0.7);
\draw (3.2,0.2)--(2.7,-2);

\draw (0,-0.3)--(0.5,-1.2);
\draw (1.3,-0.3)--(1.8,-1.2);

\node at (0.25, -1.2) {$\vdots$};

\draw (0.5,-1.6)--(0.,-2.5);

\draw (0.,-2.)--(0.5,-2.9);

\draw (0.5,-3.3)--(0,-4.2);

\draw[red,thick, dashdotted] (.2,-1.7)--(1.35,-1.74);
\draw[red,thick, dashdotted] (3,-1.8)--(1.75,-1.76);
\node[red,below] at (3,-1.8){$e_2$};

\draw (1.8,-3.3)--(1.3,-4.2);

\draw (0,-3.8)--(0.5,-4.7);
\draw (1.3,-3.8)--(1.8,-4.7);

\node at (0.25, -4.7) {$\vdots$};
\node at (1.55, -4.7) {$\vdots$};

\draw (0,-5.1)--(0.5,-6);
\draw (1.3,-5.1)--(1.8,-6);

\filldraw (1.55,-1.65) circle (0.25pt);
\filldraw (1.55,-1.75) circle (0.25pt);
\filldraw (1.55,-1.85) circle (0.25pt);
\filldraw (1.55,-1.95) circle (0.25pt);
\filldraw (1.55,-2.05) circle (0.25pt);
\filldraw (1.55,-2.15) circle (0.25pt);
\filldraw (1.55,-2.25) circle (0.25pt);
\filldraw (1.55,-2.35) circle (0.25pt);
\filldraw (1.55,-2.45) circle (0.25pt);
\filldraw (1.55,-2.55) circle (0.25pt);
\filldraw (1.55,-2.65) circle (0.25pt);
\filldraw (1.55,-2.75) circle (0.25pt);



\node at (0.25, -2.9) {$\vdots$};

\node[left] at (0,0) {$+1$};
\node[right] at (0.35, -0.25) {$-a_{11}$};
\node[right] at (1.65, -0.25) {$-a_{21}$};
\node[right] at (3.1, -0.35) {$-1$};

\node[right] at (0.35, -0.95) {$-a_{12}$};
\node[right] at (1.65, -0.95) {$-a_{22}$};

\node[right] at (0.35, -2.15) {$-a_{1l}'$};
\node[left] at (0.25, -1.65) {$C'^1_{l}$};


\node[right] at (0.25, -5.45) {$-a_{1n_1}$};
\node[right] at (1.55, -5.45) {$-a_{2n_2}$};

\end{scope}
\begin{scope}[shift={(5.2,0)}]
\draw(0,0)--(3.5,0);
\draw (0.5,0.2)--(0,-0.7);
\draw (1.8,0.2)--(1.3,-0.7);
\draw (3.2,0.2)--(2.7,-2.5);

\draw[red,thick, dashdotted](3,-2.4)--(1.75,-1.44);
\draw[red,thick, dashdotted](1.35,-1.15)--(0.1,-0.2);
\node[red,below] at (3,-2.4){$e_2$};

\draw (0,-0.3)--(0.5,-1.2);
\draw (1.3,-0.3)--(1.8,-1.2);

\node at (0.25, -1.3) {$\vdots$};
\node at (1.55, -1.3) {$\vdots$};

\draw (0.5,-2)--(0,-2.9);
\draw (1.8,-2)--(1.3,-2.9);

\draw (0.,-2.5)--(0.5,-3.4);
\draw (1.3,-2.5)--(1.8,-3.4);


\node[left] at (0,0) {$+1$};
\node[right] at (0.35, -0.25) {$-a_{11}'$};
\node[right] at (1.65, -0.25) {$-a_{21}$};
\node[right] at (3.1, -0.35) {$-1$};

\node[right] at (1.65, -0.95) {$-a_{22}$};

\draw (0,-5.1)--(0.5,-6);
\draw (1.3,-5.1)--(1.8,-6);

\node[right] at (0.25, -5.45) {$-a_{1n_1}$};
\node[right] at (1.55, -5.45) {$-a_{2n_2}$};

\filldraw (0.25,-3.7) circle (0.25pt);
\filldraw (0.25,-3.8) circle (0.25pt);
\filldraw (0.25,-3.9) circle (0.25pt);
\filldraw (0.25,-4) circle (0.25pt);
\filldraw (0.25,-4.1) circle (0.25pt);
\filldraw (0.25,-4.2) circle (0.25pt);
\filldraw (0.25,-4.3) circle (0.25pt);
\filldraw (0.25,-4.4) circle (0.25pt);
\filldraw (0.25,-4.5) circle (0.25pt);
\filldraw (0.25,-4.6) circle (0.25pt);

\filldraw (1.55,-3.7) circle (0.25pt);
\filldraw (1.55,-3.8) circle (0.25pt);
\filldraw (1.55,-3.9) circle (0.25pt);
\filldraw (1.55,-4) circle (0.25pt);
\filldraw (1.55,-4.1) circle (0.25pt);
\filldraw (1.55,-4.2) circle (0.25pt);
\filldraw (1.55,-4.3) circle (0.25pt);
\filldraw (1.55,-4.4) circle (0.25pt);
\filldraw (1.55,-4.5) circle (0.25pt);
\filldraw (1.55,-4.6) circle (0.25pt);


\end{scope}
\end{tikzpicture}
\caption{Part of configuration $C'$ obtained from (b) and (c)}
\label{caseappendix}
\end{center}
\end{figure}
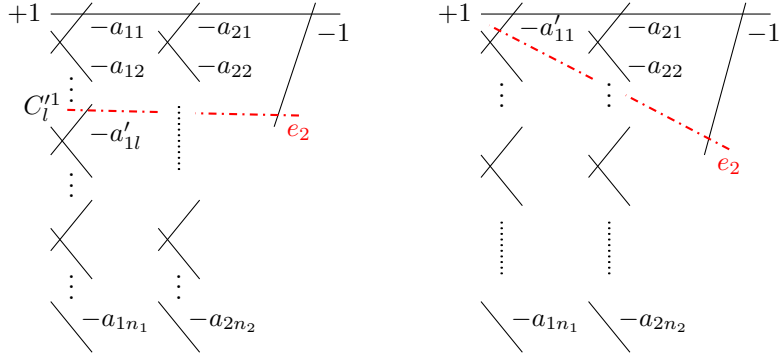
In $C'$, we have two non-trivial arms of $K$, except for one irreducible component $C'^1_l$ whose degree is $-a'_{1l}$ with $a_{1l} > a'_{1l}$, where $-a_{1l}$ is degree of the $l^{\text{th}}$ component of the first arm in $K$. To get $C$ from $C'$ by blowing-ups, we need $a_{1l}-a'_{1l}$ more blowing-ups at $C'^1_l$ to get the right degree $-a_{1l}$.  Among these blowing-ups, let $m \leq a_{1l}-a'_{1l}$ be the number of blowing-ups that occur at the intersection points of $C'^1_l$. Now we consider a plumbing graph $\Gamma_p$(refer to Figure~\ref{newgraph}) obtained from $\Gamma$ by blowing-ups at the central vertex. Let $L_h$ be a maximal horizontal subgraph of $\Gamma_p$ determined by $[-b_{1r_1},\dots, -(b+m),\dots, -b_{2r_2}]$ and $L_v$ be a vertical subgraph determined  by $ [-2,\dots,-2,-(b_{31}+1),\dots, -b_{3r_3}]$. Then we claim the following:
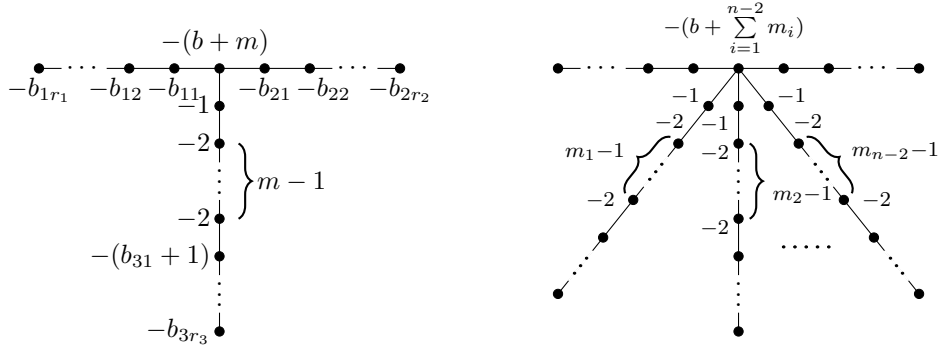
\begin{figure}
\begin{tikzpicture}[xscale=0.6, yscale=0.5]
\begin{scope}
\node[bullet] at (0,0){};

\node[bullet] at (2,0){};
\node[bullet] at (3,0){};
\node[bullet] at (4,0){};
\node[bullet] at (5,0){};
\node[bullet] at (6,0){};
\node[bullet] at (8,0){};
\node[bullet] at (4,-1){};

\node[bullet] at (4,-2){};

\node[bullet] at (4,-4){};
\node[bullet] at (4,-5){};
\node[bullet] at (4,-7){};

\node[below] at (0,0){$-b_{1r_1}$};
\node[below] at (1.7,0){$-b_{12}$};
\node[below] at (3,0){$-b_{11}$};
\node[above] at (3.9,0.1){$-(b+m)$};
\node[below] at (5,0){$-b_{21}$};
\node[below] at (6.3,0){$-b_{22}$};
\node[below] at (8,0){$-b_{2r_2}$};

\node[left] at (4,-1){$-1$};
\node[left] at (4,-2){$-2$};

\node[left] at (4,-4){$-2$};
\node[left] at (4,-5){$-(b_{31}+1)$};
\node[left] at (4,-7){$-b_{3r_3}$};

\node at (1,0){$\cdots$};
\node at (7,0){$\cdots$};
\node at (4,-2.8){$\vdots$};
\node at (4,-5.8){$\vdots$};

\draw (0,0)--(0.5,0);
\draw (1.5,0)--(6.5,0);
\draw (7.5,0)--(8,0);
\draw (4,0)--(4,-2.5);
\draw (4,-3.5)--(4,-4);
\draw (4,-4)--(4,-5.5);
\draw (4,-6.5)--(4,-7);
	\draw [thick,decorate,decoration={brace,amplitude=5pt},xshift=5pt,xshift=7pt]
	(4,-2) -- (4,-4) node [black,midway,xshift=20pt] 
	{$m-1$};
\end{scope}
\begin{scope}[shift={(11.5,0)}]
\node[bullet] at (0,0){};

\node[bullet] at (2,0){};
\node[bullet] at (3,0){};
\node[bullet] at (4,0){};
\node[bullet] at (5,0){};
\node[bullet] at (6,0){};
\node[bullet] at (8,0){};

\node[bullet] at (0,-6){};
\node[bullet] at (10/3,-1){};
\draw (0,-6)--(4,0);

\node[bullet] at (8/3,-2){};
\node[above] at (8/3-.2,-2+.1){\footnotesize$-2$};
\node[bullet] at (5/3,-3.5){};
\node[left] at (5/3-.2,-3.5){\footnotesize$-2$};
\draw[white, thick] (6.5/3+4/30+4/30,-2.35)--(6.5/3-4/30-4/30,-3.15);
\filldraw (6.5/3+4/30,-2.55) circle (0.75pt);
\filldraw (6.5/3,-2.75) circle (0.75pt);
\filldraw (6.5/3-4/30,-2.95) circle (0.75pt);
\node[bullet] at (3/3,-4.5){};
\draw[white, thick] (.5+4/30+4/30,-5.25+0.4)--(.5-4/30-4/30,-5.25-.4);
\filldraw (.5+4/30,-5.25+0.2) circle (0.75pt);
\filldraw (.5,-5.25) circle (0.75pt);
\filldraw (.5-4/30,-5.25-0.2) circle (0.75pt);

	\draw [thick,decorate,decoration={mirror,brace,amplitude=5pt},yshift=5pt,xshift=-5pt ]
	(8/3,-2) -- (5/3,-3.5) node [black,midway,xshift=-20pt, yshift=5pt] 
	{\footnotesize$m_1-$1};

\node[bullet] at (8,-6){};
\draw (8,-6)--(4,0);

\node[bullet] at (14/3,-1){};
\node[bullet] at (16/3,-2){};
\node[above] at (16/3+.2,-2+.1){\footnotesize$-2$};
\node[bullet] at (19/3,-3.5){};
\node[right] at (19/3+.2,-3.5){\footnotesize$-2$};
\draw[white, thick] (35/6-4/30-4/30,-2.35)--(35/6+4/30+4/30,-3.15);
\filldraw (35/6-4/30,-2.55) circle (0.75pt);
\filldraw (35/6,-2.75) circle (0.75pt);
\filldraw (35/6+4/30,-2.95) circle (0.75pt);
\node[bullet] at (7,-4.5){};
\draw[white, thick] (7.5-4/30-4/30,-5.25+0.4)--(7.5+4/30+4/30,-5.25-.4);
\filldraw (7.5-4/30,-5.25+0.2) circle (0.75pt);
\filldraw (7.5,-5.25) circle (0.75pt);
\filldraw (7.5+4/30,-5.25-0.2) circle (0.75pt);

	\draw [thick,decorate,decoration={brace,amplitude=5pt},yshift=5pt,xshift=5pt ]
	(16/3,-2) -- (19/3,-3.5) node [black,midway,xshift=25pt, yshift=5pt] 
	{\footnotesize$m_{n-2}-$1};

\node[bullet] at (4,-1){};

\node[bullet] at (4,-2){};

\node[bullet] at (4,-4){};
\node[bullet] at (4,-5){};
\node[bullet] at (4,-7){};

\node[above] at (3.9,0.1){\footnotesize$-(b+\sum\limits_{i=1}^{n-2}m_i)$};

\node[left] at (4,-1.5){\footnotesize$-1$};
\node[left] at (3.35,-.75){\footnotesize$-1$};
\node[right] at (4.65,-.75){\footnotesize$-1$};
\node[left] at (4,-2.25){\footnotesize$-2$};

\filldraw (5,-4.75) circle (.8pt);
\filldraw (5.25,-4.75) circle (.8pt);
\filldraw (5.5,-4.75) circle (.8pt);
\filldraw (5.75,-4.75) circle (.8pt);
\filldraw (6,-4.75) circle (.8pt);

\node[left] at (4,-4.25){\footnotesize$-2$};

\node at (1,0){$\cdots$};
\node at (7,0){$\cdots$};
\node at (4,-2.8){$\vdots$};
\node at (4,-5.8){$\vdots$};

\draw (0,0)--(0.5,0);
\draw (1.5,0)--(6.5,0);
\draw (7.5,0)--(8,0);
\draw (4,0)--(4,-2.5);
\draw (4,-3.5)--(4,-4);
\draw (4,-4)--(4,-5.5);
\draw (4,-6.5)--(4,-7);
	\draw [thick,decorate,decoration={brace,amplitude=5pt},xshift=0pt,xshift=7pt ]
	(4,-2) -- (4,-4) node [black,midway,xshift=20pt, yshift=-5pt] 
	{\footnotesize$m_2-$1};
\end{scope}
\end{tikzpicture}

\caption{A plumbing graph $\Gamma_p$}
\label{newgraph}
\end{figure}
\begin{claim}
There exist minimal symplectic fillings $W_v$ of $L_v$ and $W_h$ of $L_h$ such that $W$ is obtained from $\widetilde{W}$ by replacing $L_v$ with $W_v$ while $\widetilde{W}$ is obtained from $\Gamma_p$ by replacing $L_h$ with $W_h$.
\label{claimappendix}
\end{claim}
\begin{proof}
First, we find a minimal symplectic filling $W_h$ of $L_h$ such that a symplectic filling of $Y$ obtained from $\Gamma_p$ by replacing $L_h$ with $W_h$ is deformation equivalent to $\widetilde{W}$. For this, we consider another Seifert $3$-manifold $Y'$ with an associated plumbing graph $\Gamma'$ and its concave cap $K'$ given in Figure~\ref{new}.
\begin{figure}[h]
\begin{tikzpicture}[scale=0.6]
\begin{scope}
\node[bullet] at (0,0){};

\node[bullet] at (2,0){};
\node[bullet] at (3,0){};
\node[bullet] at (4,0){};
\node[bullet] at (5,0){};
\node[bullet] at (6,0){};
\node[bullet] at (8,0){};
\node[bullet] at (4,-1){};

\node[bullet] at (4,-2){};

\node[bullet] at (4,-4){};

\node[below] at (0,0){$-b_{1r_1}$};
\node[below] at (1.7,0){$-b_{12}$};
\node[below] at (3,0){$-b_{11}$};
\node[above] at (3.9,0.1){$-(b+m)$};
\node[below] at (5,0){$-b_{21}$};
\node[below] at (6.3,0){$-b_{22}$};
\node[below] at (8,0){$-b_{2r_2}$};

\node[left] at (4,-1.2){$-b_{31}$};
\node[left] at (4,-2.2){$-b_{32}$};

\node[left] at (4,-4){$-b_{3r_3}$};

\node at (1,0){$\cdots$};
\node at (7,0){$\cdots$};
\node at (4,-2.8){$\vdots$};

\draw (0,0)--(0.5,0);
\draw (1.5,0)--(6.5,0);
\draw (7.5,0)--(8,0);
\draw (4,0)--(4,-2.5);
\draw (4,-3.5)--(4,-4);

\end{scope}
\begin{scope}[shift={(10,0))},scale=1.8]
\draw(0,0)--(5.2,0);
\draw (0.5,0.2)--(0,-0.7);
\draw (1.5,0.2)--(1,-0.7);
\draw (2.5,0.2)--(2,-0.7);
\draw (3.9,0.2)--(3.4,-0.7);
\draw (4.9,0.2)--(4.4,-0.7);
\node at (4.,-0.6) {$\cdots$};

	\draw [decorate,decoration={brace,amplitude=5pt,mirror},xshift=2pt,yshift=-3pt]
	(3.4,-0.7) -- (4.4,-0.7) node [black,midway,yshift=-10pt] 
	{\footnotesize $b+m-4$};

\draw (0,-0.3)--(0.5,-1.2);
\draw (1,-0.3)--(1.5,-1.2);
\draw (2,-0.3)--(2.5,-1.2);

\node at (0.25, -1.5) {$\vdots$};
\node at (1.25, -1.5) {$\vdots$};
\node at (2.25, -1.5) {$\vdots$};

\draw (0,-1.8)--(0.5,-2.7);
\draw (1,-1.8)--(1.5,-2.7);
\draw (2,-1.8)--(2.5,-2.7);

\node[left] at (0,0) {$+1$};
\node[right] at (0.25, -0.25) {\footnotesize$-a_{11}$};
\node[right] at (1.25, -0.25) {\footnotesize$-a_{21}$};
\node[right] at (2.25, -0.25) {\footnotesize$-a_{31}$};
\node[right] at (3.65, -0.25) {\footnotesize$-1$};
\node[right] at (4.65, -0.25) {\footnotesize$-1$};

\node[right] at (0.25, -0.75) {\footnotesize$-a_{12}$};
\node[right] at (1.25, -0.75) {\footnotesize$-a_{22}$};
\node[right] at (2.25, -0.75) {\footnotesize$-a_{32}$};

\node[right] at (0.25, -2.25) {\footnotesize$-a_{1n_1}$};
\node[right] at (1.25, -2.25) {\footnotesize$-a_{2n_2}$};
\node[right] at (2.25, -2.25) {\footnotesize$-a_{3n_3}$};

\end{scope}
\end{tikzpicture}
\caption{A plumbing graph $\Gamma'$ and concave cap $K'$}
\label{new}
\end{figure}
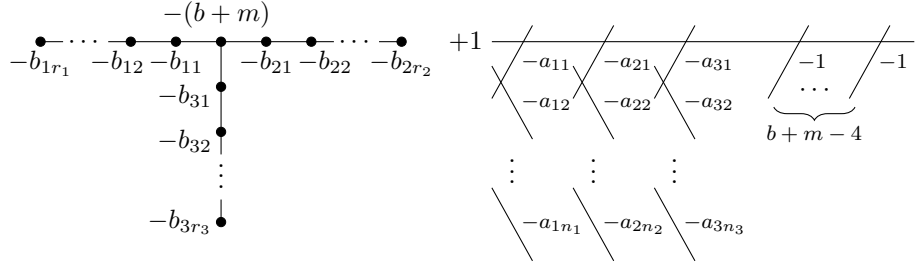 
Note that $K\subset K'$ and there is a (-1) curve connecting the central curve of $\Gamma'$ and each single $(-1)$ arm of $K'$ in the rational surface $(\Gamma' \cup K')$. Furthermore, by blowing down $m$ such $(-1)$ curves, we get $\Gamma$ together with $K$ so that a non-minimal symplectic filling $\Gamma_p$ of $Y$ is deformation equivalent to $(\Gamma' \cup K')\setminus K$ and $L_h\subset \Gamma_p$ is isotopic to maximal horizontal subgraph of $\Gamma'$ which also denoted by $L_h$.
Now we construct a desired minimal symplectic filling $W_h$ of $L_h$ using a sequence of blowing-ups from a symplectic line arrangement to $C'$. Instead of a symplectic line arrangement with $(b-1)$ lines, we start from a symplectic line arrangement with $(b+m-1)$ lines. Then, by using a sequence of blowing-ups to $C'$, we get a configuration $C''$ as in Figure~\ref{caseappendix2}. Note that $C''$ differs from $C'$ by a number of single $(-1)$ arms and degree of the $l^\text{th}$ component of the first arm. Precisely, the difference between the degrees of two components is exactly $m$ coming from $(-1)$ curves connecting the component and $m$ more single $(-1)$ arms in $C''$. 
\begin{figure}[h]
\begin{center}
\begin{tikzpicture}[xscale=1.1,yscale=.75]
\begin{scope}
\draw(0,0)--(4.2,0);
\draw (0.5,0.2)--(0,-0.7);
\draw (1.8,0.2)--(1.3,-0.7);
\draw (3.2,0.2)--(2.7,-2);
\draw (3.9,0.2)--(3.4,-2);

	\draw [decorate,decoration={brace,amplitude=5pt},xshift=0pt,yshift=-3pt]
	(3.4,-2) -- (2.7,-2) node [black,midway,yshift=-10pt] 
	{\footnotesize $m$};

\filldraw (3.2-.15,-1.2) circle (0.25pt);
\filldraw (3.3-.15,-1.2) circle (0.25pt);
\filldraw (3.4-.15,-1.2) circle (0.25pt);
\filldraw (3.5-.15,-1.2) circle (0.25pt);
\filldraw (3.6-.15,-1.2) circle (0.25pt);

\draw (0,-0.3)--(0.5,-1.2);
\draw (1.3,-0.3)--(1.8,-1.2);

\node at (0.25, -1.2) {$\vdots$};

\draw (0.5,-1.6)--(0.,-2.5);

\draw (0.,-2.)--(0.5,-2.9);

\draw (0.5,-3.3)--(0,-4.2);

\draw[red,thick, dashdotted] (.2,-1.7)--(1.35,-1.74);
\draw[red,thick, dashdotted] (3,-1.8)--(1.75,-1.76);

\draw (1.8,-3.3)--(1.3,-4.2);

\draw (0,-3.8)--(0.5,-4.7);
\draw (1.3,-3.8)--(1.8,-4.7);

\node at (0.25, -4.7) {$\vdots$};
\node at (1.55, -4.7) {$\vdots$};

\draw (0,-5.1)--(0.5,-6);
\draw (1.3,-5.1)--(1.8,-6);

\filldraw (1.55,-1.55) circle (0.25pt);
\filldraw (1.55,-1.65) circle (0.25pt);
\filldraw (1.55,-1.75) circle (0.25pt);
\filldraw (1.55,-1.85) circle (0.25pt);
\filldraw (1.55,-2.55) circle (0.25pt);
\filldraw (1.55,-2.65) circle (0.25pt);
\filldraw (1.55,-2.75) circle (0.25pt);
\filldraw (1.55,-2.85) circle (0.25pt);



\node at (0.25, -2.9) {$\vdots$};

\node[left] at (0,0) {$+1$};
\node[right] at (0.35, -0.25) {$-a_{11}$};
\node[right] at (1.65, -0.25) {$-a_{21}$};
\node[right] at (3.1, -0.35) {$-1$};
\node[right] at (3.7, -0.35) {$-1$};
\node[right] at (0.35, -0.95) {$-a_{12}$};
\node[right] at (1.65, -0.95) {$-a_{22}$};

\node[right] at (0.35, -2.25) {$-(a_{1l}'+m)$};


\node[right] at (0.25, -5.45) {$-a_{1n_1}$};
\node[right] at (1.55, -5.45) {$-a_{2n_2}$};

\end{scope}
\begin{scope}[shift={(6.2,0)}]
\draw(0,0)--(4.2,0);
\draw (0.5,0.2)--(0,-0.7);
\draw (1.8,0.2)--(1.3,-0.7);
\draw (3.2,0.2)--(2.7,-2.5);
\draw (3.9,0.2)--(3.4,-2.5);

	\draw [decorate,decoration={brace,amplitude=5pt},xshift=0pt,yshift=-3pt]
	(3.4,-2.5) -- (2.7,-2.5) node [black,midway,yshift=-10pt] 
	{\footnotesize $m$};

\filldraw (3.2-.15,-1.5) circle (0.25pt);
\filldraw (3.3-.15,-1.5) circle (0.25pt);
\filldraw (3.4-.15,-1.5) circle (0.25pt);
\filldraw (3.5-.15,-1.5) circle (0.25pt);
\filldraw (3.6-.15,-1.5) circle (0.25pt);

\draw[red,thick, dashdotted](3,-2.4)--(1.75,-1.44);
\draw[red,thick, dashdotted](1.35,-1.15)--(0.1,-0.2);

\draw (0,-0.3)--(0.5,-1.2);
\draw (1.3,-0.3)--(1.8,-1.2);

\node at (0.25, -1.3) {$\vdots$};
\node at (1.55, -1.3) {$\vdots$};

\draw (0.5,-2)--(0,-2.9);
\draw (1.8,-2)--(1.3,-2.9);

\draw (0.,-2.5)--(0.5,-3.4);
\draw (1.3,-2.5)--(1.8,-3.4);


\node[left] at (0,0) {$+1$};
\node[] at (-.7, -1) {$-(a_{11}'+m)$};
\node[right] at (1.65, -0.25) {$-a_{21}$};
\node[right] at (3.1, -0.35) {$-1$};
\node[right] at (3.8, -0.35) {$-1$};

\node[right] at (1.65, -0.95) {$-a_{22}$};

\draw (0,-5.1)--(0.5,-6);
\draw (1.3,-5.1)--(1.8,-6);

\node[right] at (0.25, -5.45) {$-a_{1n_1}$};
\node[right] at (1.55, -5.45) {$-a_{2n_2}$};

\filldraw (0.25,-3.7) circle (0.25pt);
\filldraw (0.25,-3.8) circle (0.25pt);
\filldraw (0.25,-3.9) circle (0.25pt);
\filldraw (0.25,-4) circle (0.25pt);
\filldraw (0.25,-4.1) circle (0.25pt);
\filldraw (0.25,-4.2) circle (0.25pt);
\filldraw (0.25,-4.3) circle (0.25pt);
\filldraw (0.25,-4.4) circle (0.25pt);
\filldraw (0.25,-4.5) circle (0.25pt);
\filldraw (0.25,-4.6) circle (0.25pt);

\filldraw (1.55,-3.7) circle (0.25pt);
\filldraw (1.55,-3.8) circle (0.25pt);
\filldraw (1.55,-3.9) circle (0.25pt);
\filldraw (1.55,-4) circle (0.25pt);
\filldraw (1.55,-4.1) circle (0.25pt);
\filldraw (1.55,-4.2) circle (0.25pt);
\filldraw (1.55,-4.3) circle (0.25pt);
\filldraw (1.55,-4.4) circle (0.25pt);
\filldraw (1.55,-4.5) circle (0.25pt);
\filldraw (1.55,-4.6) circle (0.25pt);


\end{scope}
\end{tikzpicture}
\caption{Part of intermediate configuration $C''$}
\label{caseappendix2}
\end{center}
\end{figure}
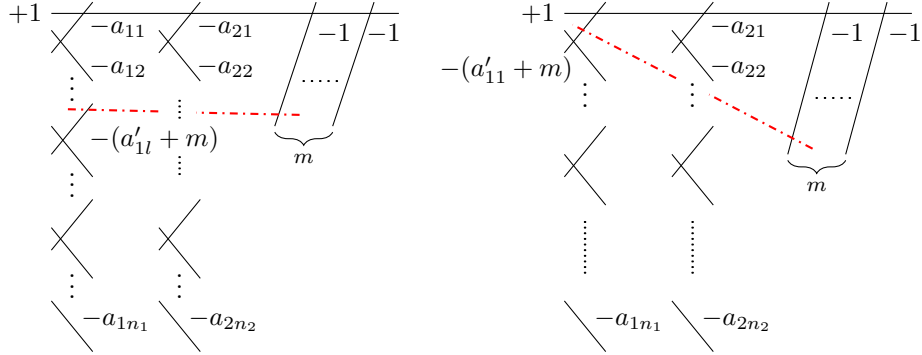
Consider a curve configuration $\widetilde{C''}$ of $Y'$ which is obtained from $C''$ by standard blowing-ups. 
Then, there exists a minimal symplectic filling $W_h$ of $L_h$ such that a minimal symplectic filling $W'$ corresponding to $\widetilde{C''}$ is deformation equivalent to $(\Gamma' \setminus L_h) \cup W_h$.
Furthermore, since the only difference between $K$ and $K'$ is the number of $(-1)$ single arms, the homological data of $K\subset K'$ in $\widetilde{C''}$ is exactly the same as that of $K$ in $\widetilde{C'}$. Therefore, $(W' \cup K') \setminus K$ is deformation equivalent to $\widetilde{W}$, so that $\widetilde{W}$ is deformation equivalent to $(\Gamma_p  \setminus L_h) \cup W_h$.


It remains to show that the aforementioned linear chain $L\subset \widetilde{W}$ for $W$ is isotopic to $L_v$ in $ \widetilde{W} \cong (\Gamma_p \setminus L_h)\cup W_h $. As we saw in Section~\ref{sec-3}, any symplectically embedded linear chain in a minimal symplectic filling is obtained from an exceptional $2$-sphere by blowing-ups. Therefore, in order to show that $L$ is isotopic to $L_v$, we only need to compare their homological data in $\widetilde{C'}$ and $\widetilde{C''}$. From the proof of Lemma 4.3 in ~\cite{CP2}, we know that $L$ is obtained from $e_2$ of $C'$ by blowing-ups as in Figure~\ref{detail4.2}.  
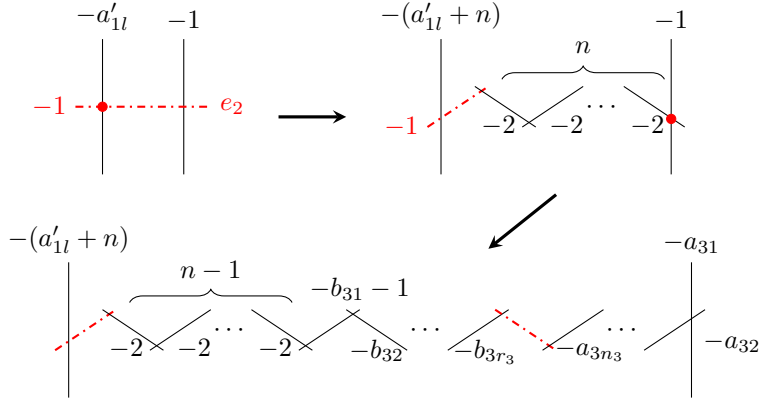
\begin{figure}[h]
\centering
\begin{tikzpicture}[scale=.9]
\begin{scope}

\draw (0.4,0.5)--(0.4,-1.5);

\draw (1.6,0.5)--(1.6,-1.5);
\draw[red,thick, dashdotted] (0,-0.5)--(2,-0.5);
\node[left,red] at (0,-0.5) {$-1$};
\node[red,right] at (2,-0.5) {$e_2$};

\node[bullet,red] at (0.4,-.5){};

\node[above] at (0.4,0.5) {$-a_{1l}'$};

\node[above] at (1.6,0.5) {$-1$};

\draw[very thick, ->, >=stealth] (3,-0.65)--(4,-0.65);
\end{scope}
\begin{scope}[shift={(5,0)}]
\draw (0.4,0.5)--(0.4,-1.5);

\draw (3.8,0.5)--(3.8,-1.5);
\node[left,red] at (0.2,-0.8) {$-1$};
\draw[red,thick, dashdotted] (0.2,-0.8)--(1.1,-0.2);

\draw (0.9,-0.2)--(1.8,-0.8);
\node[below] at (1.25,-.5){$-2$};
\draw (1.6,-0.8)--(2.5,-0.2);
\node[below] at (2.25,-.5){$-2$};
\node at (2.8,-.5){$\cdots$};
\draw (3.1,-0.2)--(4,-0.8);
\node[below] at (3.45,-.5){$-2$};

\node[bullet,red] at (3.8,-0.68){};

\node[above] at (0.4,0.5) {$-(a'_{1l}+n)$};

\node[above] at (3.8,0.5) {$-1$};

\draw [decorate,decoration={brace, amplitude=5pt},xshift=0pt,yshift=3pt]
	(1.3,-0.2) -- (3.7,-.2) node [black,midway,yshift=12pt] 
	{$n$};
\draw[very thick, ->, >=stealth] (2.1,-1.8)--(1.1,-2.6);

\end{scope}
\begin{scope}[shift={(-.5,-3.3)}]
\draw (0.4,0.5)--(0.4,-1.5);

\draw (9.6,0.5)--(9.6,-1.5);
\draw[red,thick, dashdotted] (0.2,-0.8)--(1.1,-0.2);

\draw (0.9,-0.2)--(1.8,-0.8);
\node[below] at (1.25,-.5){$-2$};
\draw (1.6,-0.8)--(2.5,-0.2);
\node[below] at (2.25,-.5){$-2$};

\node at (2.8,-.5){$\cdots$};
\draw (3.1,-0.2)--(4,-0.8);
\node[below] at (3.45,-.5){$-2$};

\draw (3.8,-0.8)--(4.7,-0.2);
\node[above] at (4.7,-.2){$-b_{31}-1$};

\draw (4.5,-0.2)--(5.4,-0.8);
\node[below] at (4.95,-.5){$-b_{32}$};
\node at (5.7,-.5){$\cdots$};
\draw (6,-.8)--(6.9,-.2);
\node[below] at (6.55,-.5){$-b_{3r_3}$};

\draw[red,thick, dashdotted] (6.7,-.2)--(7.6,-.8);

\draw (7.4,-.8)--(8.3,-.2);
\node[below] at (8.1,-.5){$-a_{3n_3}$};

\node at (8.6,-.5){$\cdots$};

\draw (8.9,-.8)--(9.8,-.2);
\node[below] at (10.2,-0.4) {$-a_{32}$};

\node[above] at (0.4,0.5) {$-(a'_{1l}+n)$};

\node[above] at (9.6,0.5) {$-a_{31}$};

\draw [decorate,decoration={brace, amplitude=5pt},xshift=0pt,yshift=3pt]
	(1.3,-0.2) -- (3.7,-.2) node [black,midway,yshift=12pt] 
	{$n-1$};

\end{scope}

\end{tikzpicture}

\caption{Embedding of $L$ to $\widetilde{W}$}
\label{detail4.2}
\end{figure}
In particular, the homological data for $[-2,\dots,-2 ]$ in $L$ is given by $(-1)$ curves only intersecting $C^1_l$ of $K$ in $\widetilde{C'}$. On the other hand, the homological data of $[-2,\dots,-2]$ in $L_v$ with respect to $\Gamma_p \cong (\Gamma'\cup K')\setminus K$ is given by (-1) curves connecting the central curve of $\Gamma'$ and each single $(-1)$ arm of $K'\setminus K$.
Hence the homological data of  $[-2,\dots,-2]$ in $L_v$ with respect to $(\Gamma_p \setminus L_h)\cup W_h$ is given by $(-1)$ curves connecting $C^1_l$ of $K'$ and single $(-1)$ arms of $K'\setminus K$ in $\widetilde{C''}$, which are $(-1)$ curves only intersecting $C^1_l$ from the viewpoint of $K$. Clearly $[-b_{32},\dots, -b_{3r_3}]$ part has the same homological data, so that we are done.

\end{proof}
In summary, using explicit one-to-one correspondences between minimal symplectic fillings and $P$-resolutions of cyclic quotient surface singularities, we get a partial resolution $(Z, E)$ corresponding to $W$, whose resolution graph is obtained from $\Gamma_p$ by blowing-ups for the 3-legged case up to now. In general case, i.e.,  $\Gamma$ has more than 3-legs (refer to Figure~\ref{newgraph}), the only difference between 3-legged case and  general case is that we get a sequence of non-negative integers $(m_1,\dots, m_{n-2})$ for the rest of arms of $K$ instead of a single $m$ for the third arm of $K$ via blowing-ups from $C'$ to $C$. Hence the same argument works for general cases, showing that there is a partial resolution $(Z,E)$ corresponding  to $W$.
\end{proof}
To complete the proof of Theorem~\ref{a}, it remains to check the ample condition on $f: (Z,E)\rightarrow (X,0)$.
\begin{prop}
The partial resolution $f: (Z,E)\rightarrow (X,0)$ in Proposition ~\ref{partial} satisfies the ample condition, that is, $K_Z$ is ample relative to $f$.
\label{ample}
\end{prop}
 \begin{proof}
Recall that the ample condition is equivalent to the discrepancy condition on each $(-1)$ curve on $\widetilde{Z}$, where $\widetilde{Z}$ is the minimal resolution of the partial resolution $Z$. For a partial resolution $(Z, E)\rightarrow (X,0)$ from minimal symplectic fillings of type A or B, every $(-1)$ curve in $\widetilde{Z}$ comes from a $P$-resolution of a cyclic quotient singularity. Hence the discrepancy condition for type A and B is satisfied. 

To check the type C case, we start with a 3-legged case as before. Note that the $(-1)$ curve in $\Gamma_p$ of Figure~\ref{newgraph} becomes the only $(-1)$ curve in $\widetilde{Z}$ not coming from a $P$-resolution of a cyclic quotient singularity. From the previous construction of our partial resolution, the $(-1)$ curve in $\widetilde{Z}$ connects two singularities of class $T$, whose corresponding continued fractions are $[c_1, \dots, c_t, \dots, c_r]$ and $[2, \dots, 2, d_1, \dots, d_s]$, as in the Figure~\ref{amplecondition}.  Therefore it suffices to show that the sum of discrepancies of the $(-c_t)$ curve and the first $(-2)$ curve (or the first $(-d_1)$ curve in case of $m=1$) is less than $-1$. 
\begin{figure}
\begin{tikzpicture}[xscale=0.6, yscale=0.5]
\begin{scope}

\draw (0,0)--(0.5,0);
\draw (1.5,0)--(6.5,0);
\draw (7.5,0)--(8,0);
\draw (4,0)--(4,-2.5);
\draw (4,-3.5)--(4,-4);
\draw (4,-4)--(4,-5.5);
\draw (4,-6.5)--(4,-7);

\node [draw, fill=white, shape=rectangle, anchor=center] at (0,0) {};
\node [draw, fill=white, shape=rectangle, anchor=center] at (2,0) {};
\node [draw, fill=white, shape=rectangle, anchor=center] at (3,0) {};
\node [draw, fill=white, shape=rectangle, anchor=center] at (4,0) {};
\node [draw, fill=white, shape=rectangle, anchor=center] at (5,0) {};
\node [draw, fill=white, shape=rectangle, anchor=center] at (6,0) {};
\node [draw, fill=white, shape=rectangle, anchor=center] at (8,0) {};

\node [draw, fill=white, shape=rectangle, anchor=center] at (4,-2) {};
\node [draw, fill=white, shape=rectangle, anchor=center] at (4,-4) {};
\node [draw, fill=white, shape=rectangle, anchor=center] at (4,-5) {};
\node [draw, fill=white, shape=rectangle, anchor=center] at (4,-7) {};

\node[bullet] at (4,-1){};

\node[above] at (0,0.1){$-c_1$};
\node[above] at (4,0.1){$-c_t$};
\node[above] at (8,0.1){$-c_r$};


\node[left] at (3.8,-1){$-1$};
\node[left] at (3.8,-2){$-2$};

\node[left] at (3.8,-4){$-2$};
\node[left] at (3.8,-5){$-d_1$};
\node[left] at (3.8,-7){$-d_s$};

\node at (1,0){$\cdots$};
\node at (7,0){$\cdots$};
\node at (4,-2.8){$\vdots$};
\node at (4,-5.8){$\vdots$};

	\draw [thick,decorate,decoration={brace,amplitude=5pt},xshift=5pt,xshift=7pt]
	(4,-2) -- (4,-4) node [black,midway,xshift=20pt] 
	{$m-1$};
	\end{scope}
\begin{scope}[shift={(11.5,0)}]
\node[bullet] at (10/3,-1){};
\draw (0,-6)--(4,0);

\node[draw, fill=white, shape=rectangle, anchor=center] at (8/3,-2){};
\node[above] at (8/3-.2,-2+.1){\footnotesize$-2$};
\node[draw, fill=white, shape=rectangle, anchor=center] at (5/3,-3.5){};
\node[left] at (5/3-.2,-3.5){\footnotesize$-2$};
\draw[white, thick] (6.5/3+4/30+4/30,-2.35)--(6.5/3-4/30-4/30,-3.15);
\filldraw (6.5/3+4/30,-2.55) circle (0.75pt);
\filldraw (6.5/3,-2.75) circle (0.75pt);
\filldraw (6.5/3-4/30,-2.95) circle (0.75pt);
\node[draw, fill=white, shape=rectangle, anchor=center] at (3/3,-4.5){};
\draw[white, thick] (.5+4/30+4/30,-5.25+0.4)--(.5-4/30-4/30,-5.25-.4);
\filldraw (.5+4/30,-5.25+0.2) circle (0.75pt);
\filldraw (.5,-5.25) circle (0.75pt);
\filldraw (.5-4/30,-5.25-0.2) circle (0.75pt);

	\draw [thick,decorate,decoration={mirror,brace,amplitude=5pt},yshift=5pt,xshift=-10pt ]
	(8/3,-2) -- (5/3,-3.5) node [black,midway,xshift=-20pt, yshift=5pt] 
	{\footnotesize$m_1-$1};

\draw (8,-6)--(4,0);
\node[draw, fill=white, shape=rectangle, anchor=center] at (8,-6){};
\node[bullet] at (14/3,-1){};
\node[draw, fill=white, shape=rectangle, anchor=center] at (16/3,-2){};
\node[above] at (16/3+.2,-2+.1){\footnotesize$-2$};
\node[draw, fill=white, shape=rectangle, anchor=center] at (19/3,-3.5){};
\node[right] at (19/3+.2,-3.5){\footnotesize$-2$};
\draw[white, thick] (35/6-4/30-4/30,-2.35)--(35/6+4/30+4/30,-3.15);
\filldraw (35/6-4/30,-2.55) circle (0.75pt);
\filldraw (35/6,-2.75) circle (0.75pt);
\filldraw (35/6+4/30,-2.95) circle (0.75pt);
\node[draw, fill=white, shape=rectangle, anchor=center] at (7,-4.5){};
\draw[white, thick] (7.5-4/30-4/30,-5.25+0.4)--(7.5+4/30+4/30,-5.25-.4);
\filldraw (7.5-4/30,-5.25+0.2) circle (0.75pt);
\filldraw (7.5,-5.25) circle (0.75pt);
\filldraw (7.5+4/30,-5.25-0.2) circle (0.75pt);

	\draw [thick,decorate,decoration={brace,amplitude=5pt},yshift=5pt,xshift=10pt ]
	(16/3,-2) -- (19/3,-3.5) node [black,midway,xshift=25pt, yshift=5pt] 
	{\footnotesize$m_{n-2}-$1};

\node[above] at (0,0.1){$-c_1$};
\node[above] at (4,0.1){$-c_t$};
\node[above] at (8,0.1){$-c_r$};


\node[left] at (4,-1.5){\footnotesize$-1$};
\node[left] at (3.35,-.75){\footnotesize$-1$};
\node[right] at (4.65,-.75){\footnotesize$-1$};
\node[left] at (4,-2.25){\footnotesize$-2$};

\filldraw (5,-4.75) circle (.8pt);
\filldraw (5.25,-4.75) circle (.8pt);
\filldraw (5.5,-4.75) circle (.8pt);
\filldraw (5.75,-4.75) circle (.8pt);
\filldraw (6,-4.75) circle (.8pt);

\node[left] at (4,-4.25){\footnotesize$-2$};

\node at (1,0){$\cdots$};
\node at (7,0){$\cdots$};
\node at (4,-2.8){$\vdots$};
\node at (4,-5.8){$\vdots$};

\draw (0,0)--(0.5,0);
\draw (1.5,0)--(6.5,0);
\draw (7.5,0)--(8,0);
\draw (4,0)--(4,-2.5);
\draw (4,-3.5)--(4,-4);
\draw (4,-4)--(4,-5.5);
\draw (4,-6.5)--(4,-7);
	\draw [thick,decorate,decoration={brace,amplitude=5pt},xshift=5pt,xshift=7pt ]
	(4,-2) -- (4,-4) node [black,midway,xshift=15pt, yshift=-5pt] 
	{\footnotesize$m_2-$1};

\node [bullet] at (4,-1) {};
\node [draw, fill=white, shape=rectangle, anchor=center] at (4,-2) {};
\node [draw, fill=white, shape=rectangle, anchor=center] at (4,-4) {};
\node [draw, fill=white, shape=rectangle, anchor=center] at (4,-5) {};
\node [draw, fill=white, shape=rectangle, anchor=center] at (4,-7) {};

\node[draw, fill=white, shape=rectangle, anchor=center] at (0,0){};

\node[draw, fill=white, shape=rectangle, anchor=center] at (2,0){};
\node[draw, fill=white, shape=rectangle, anchor=center] at (3,0){};
\node[draw, fill=white, shape=rectangle, anchor=center] at (4,0){};
\node[draw, fill=white, shape=rectangle, anchor=center] at (5,0){};
\node[draw, fill=white, shape=rectangle, anchor=center] at (6,0){};
\node[draw, fill=white, shape=rectangle, anchor=center] at (8,0){};

\node[draw, fill=white, shape=rectangle, anchor=center] at (0,-6){};

\end{scope}

\end{tikzpicture}
\caption{$(-1)$ curve connecting two singularities in $\widetilde{Z}$}
\label{amplecondition}
\end{figure}
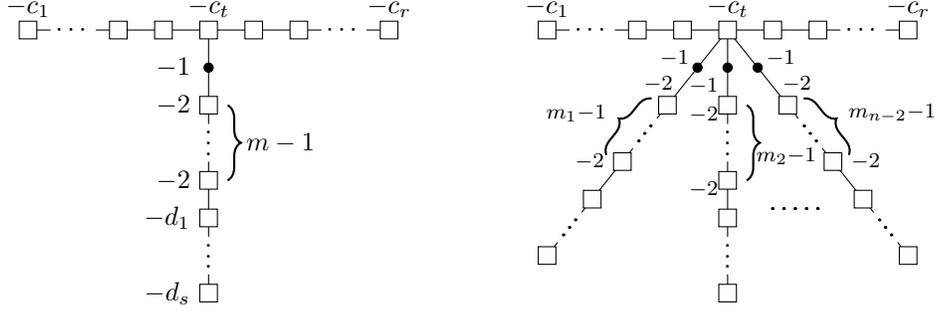

Without loss of generality, we can assume that the two $T$-singularities are actually Wahl singularities, whose corresponding continued fractions are obtained from $[4]$, because there is a unique $M$-resolution dominating a $P$-resolution of a cyclic quotient surface singularity~\cite{BC}. To show a desired inequality for the sum of discrepancies, we use an inductive description for discrepancies of Wahl singularities introduced in \cite{UV}: Let $[b_1, \dots, b_r]$ be a continued fraction corresponding to a Wahl singularity. Since the continued fraction is obtained from a single $[4]$ inductively (See, the Proposition~\ref{inductiveofsingT}), its discrepancy can also be computed inductively. 
We define a $\delta$-sequence $(\delta_1, \dots, \delta_r)$ of integers corresponding to a Wahl continued fraction $[b_1, \dots, b_r]$ inductively as follows.
\begin{enumerate}[(i)]
\item (1) corresponds to $[4]$
\item If $(\delta_1, \dots, \delta_r)$ corresponds to $[b_1,\dots, b_r]$, then 
\begin{itemize}
\item $(\delta_1, \dots, \delta_r, \delta_1 + \delta_r) \text{ corresponds to } [b_1+1, b_2, \dots, b_r, 2]$ and
\item $(\delta_1 + \delta_r, \delta_1, \dots, \delta_r) \text{ corresponds to } [2, b_1, \dots, b_{r-1}, b_r+1].$
\end{itemize}
\end{enumerate}
Then the discrepancy $m_i$ of a $(-b_i)$ curve is equal to $\left(-1 + \frac{\delta_i}{\delta_1 + \delta_r}\right)$.

First, we find a bound for the discrepancy of a $(-c_t)$ curve in the Wahl singularity corresponding to $[c_1,\dots, c_t, \dots, c_r]$ with $1<t<r$.
\begin{lemma}
Let $[c_1, \dots, c_t, \dots, c_r]$ be a continued fraction corresponding to a Wahl singularity with $c_t \geq 5$. Then the discrepancy $m_t$ of $c_t$ is less than or equal to $-1+ \frac{1}{c_t}$.
\label{D^2<-4}
\end{lemma}
\begin{proof}
Let $Z$ be a Wahl singularity corresponding to the given fraction. Then $K_{\widetilde{Z}} = \pi^*K_Z + \sum\limits_{i=1}^{r} m_i E_i$, where $E_i^2 = -c_i$. By multiplying an exceptional curve $E_t$, we obtain $-2 + c_t = m_{t-1} + m_{t+1} -m_t c_t$, so that $m_t = -1+\frac{2+m_{t-1}+m_{t+1}}{c_t}$. Consequently, it suffices to show that $m_{t-1} + m_{t+1} \leq -1$. 

First, we assume that the $(-c_t)$ curve is an initial curve of $Z$, that is, $c_t$ in $[c_1,\dots, c_t, \dots, c_r]$ comes from 5 of $[3, 5, 2]$ or $[2, 5, 3]$, whose corresponding $\delta$ sequence is $(2, 1, 3)$ or $(3, 1, 2)$, under inductive steps from $[4]$ to $[c_1,\dots, c_r]$. Let $(\delta_1, \cdots, \delta_r)$ be a $\delta$-sequence corresponding to $[c_1, \cdots, c_t, \cdots, c_r]$. Then we get $\delta_{t-1}+\delta_{t+1}=2+3=5$ and  $\delta_{1}+\delta_{r}\geq2+3=5$ from the inductive definition of $\delta$-sequence. Therefore $m_{t-1} + m_{t+1} = \left(-1 + \frac{\delta_{t-1}}{\delta_1 + \delta_r} \right) + \left(-1 + \frac{\delta_{t+1}}{\delta_1 + \delta_r} \right) = \left(-2 + \frac{\delta_{t-1} + \delta_{t+1}}{\delta_1 + \delta_r} \right) \leq -2 + \frac{5}{5} = -1$.


Secondly, we assume that the $(-c_t)$ curve is not an initial curve. Then we have the following inductive steps from $[4]$ to $[c_1,\dots, c_r]$:
$$[4] \rightarrow [ \dots, 2] \rightarrow [2,\dots, 2, \dots, c_t] \rightarrow [3,\dots,2,\dots, c_t,2] \rightarrow [c_1,\dots, c_r].$$
Then a $\delta$-sequence of the second continued fraction is $(\delta_1,\!\dots\!, \delta_{t'}, \delta_1\!+\!\delta_{t'})$, so that we have a $\delta$-sequence $((c_t\!-\!1)\delta_1 + (c_t\!-\!2)\delta_{t'}, \cdots, \delta_{t'}, \delta_1 + \delta_{t'}, c_t\delta_1+(c_t\!-\!1)\delta_{t'})$ for the fourth continued fraction.
Therefore, $m_{t-1} + m_{t+1} \leq \left(-2 + \frac{c_t\delta_1+c_t\delta_{t'}}{(2c_t-1)\delta_1 + (2c_t-3)\delta_{t'}} \right) < -1$.
\end{proof}

Next, we find bounds for the first curve of $[2, \dots, 2, d_1,\dots , d_s]$ and the $(-c_t)$ curve when $t=1$ or $r$, by using a lemma regarding discrepancies given in \cite{UV}.

\begin{lemma}[\cite{UV}, Lemma 4.4]
Let $[b_1, \dots, b_r]$ be a Wahl continued fraction, assume $r \geq 2$ and $b_r = 2$, and let us denote its discrepancies by $m_1, \dots, m_r$. Then we have the following bounds:\\
(Type M) If $b_2 = b_3 = \cdots = b_r$, then $m_1 = -1 + \frac{1}{b_1-2}$ and $m_r = -\frac{1}{b_1-2}$.\\
(Type B) Otherwise, $m_1 = -1 + \mu$ and $m_r = -\mu$, where $\frac{1}{b_1} < \mu < \frac{1}{b_1 - 1}$.
\label{lemma44inUV}
\end{lemma}

Using the lemmas above, we get that the discrepancy of a $(-c_t)$ curve is less than or equal to $-1+\frac{1}{c_t-2}$ while the discrepancy of the first curve of $[2,\dots,2,d_1,\dots, d_s]$ is less than $-\frac{1}{m+1}$. Since $b\geq 5$, we have $c_t\geq m+5$. Hence the sum of two discrepancies is less than $(-1+\frac{1}{c_t-2} -\frac{1}{m+1}) \leq (-1+\frac{1}{m+3} -\frac{1}{m+1})<-1$.

For an $n$-legged case, there are at most $(n-2)$ many $(-1)$ curves in $\widetilde{Z}$ not coming from $P$-resolutions of cyclic quotient singularities (refer to Figure~\ref{amplecondition}). Note that such a $(-1)$ curve intersects the central $(-c_t)$ curve of a Wahl singularity $[c_1, \dots, c_t, \dots, c_r]$
and the first curve of a Wahl singularity of the form $[2, 2, \dots, 2, \dots]$, where the number of consecutive $2$ is $(m_i-1)$. Since $c_t \geq n+2+\sum\limits_{i=1}^{n-2} m_i $ with $n \geq 3$, the sum of two discrepancies is less than $(-1+\frac{1}{c_t-2} -\frac{1}{m_i+1}) <-1$ for each $(-1)$ curve, which proves the ample condition.




 \end{proof}




\bigskip

\providecommand{\bysame}{\leavevmode\hbox to3em{\hrulefill}\thinspace}

\end{document}